\documentclass[reqno]{amsart}
\usepackage{textcomp}

\usepackage{amsmath}
\usepackage{amsthm}
\usepackage{hyperref}
\usepackage{amsfonts,graphics,amsthm,amsfonts,amscd,latexsym}
\usepackage{tikz-cd}
\usepackage{epsfig}
\usepackage{flafter}
\usepackage{longtable}
\usepackage{mathtools}
\usepackage{comment}
\usepackage{stmaryrd}

\usepackage{mathabx,epsfig}

\hypersetup{
    colorlinks=true,    
    linkcolor=blue,          
    citecolor=blue,      
    filecolor=blue,      
    urlcolor=blue           
}
\usepackage{tikz}
\usetikzlibrary{graphs,positioning,arrows,shapes.misc,decorations.pathmorphing}

\tikzset{
    >=stealth,
    every picture/.style={thick},
    graphs/every graph/.style={empty nodes},
}

\tikzstyle{vertex}=[
    draw,
    circle,
    fill=black,
    inner sep=1pt,
    minimum width=5pt,
]
\usepackage[position=top]{subfig}
\usepackage{amssymb}
\usepackage{color}

\calclayout

\usetikzlibrary{decorations.pathmorphing}
\tikzstyle{printersafe}=[decoration={snake,amplitude=0pt}]

\newcommand{\pp}{\mathbb{P}}

\newcommand{\qq}{\mathbb{Q}}

\newcommand{\cc}{\mathbb{C}}

\newcommand{\oo}{\mathcal{O}}

\def\O#1.{\mathcal {O}_{#1}}			
\def\pr #1.{\mathbb P^{#1}}				
\def\af #1.{\mathbb A^{#1}}			
\def\ses#1.#2.#3.{0\to #1\to #2\to #3 \to 0}	
\def\xrar#1.{\xrightarrow{#1}}			
\def\K#1.{K_{#1}}						
\def\bA#1.{\mathbf{A}_{#1}}			
\def\bM#1.{\mathbf{M}_{#1}}				
\def\bL#1.{\mathbf{L}_{#1}}				
\def\bB#1.{\mathbf{B}_{#1}}				
\def\bK#1.{\mathbf{K}_{#1}}			
\def\subs#1.{_{#1}}					
\def\sups#1.{^{#1}}

\usepackage{tikz}
\usetikzlibrary{matrix,arrows,decorations.pathmorphing}

\newtheorem{introthm}{Theorem}

  \newtheorem{theorem}{Theorem}[section]
  \newtheorem{lemma}[theorem]{Lemma}
  \newtheorem{proposition}[theorem]{Proposition}

  \newtheorem{definition}[theorem]{Definition}
  \newtheorem{example}[theorem]{Example}
  
  \newtheorem{problem}[theorem]{Problem}
  \newtheorem{question}[theorem]{Question}

\theoremstyle{remark}

\numberwithin{equation}{section}

\usepackage[all]{xy}

\begin{document}

\title[Toric models of smooth Fano threefolds]{Toric models of smooth Fano threefolds}

\author[K.~Loginov]{Konstantin Loginov}
\address{Steklov Mathematical Institute of Russian Academy of Sciences, Moscow, Russia.
}
\email{loginov@mi-ras.ru}

\author[J.~Moraga]{Joaqu\'in Moraga}
\address{UCLA Mathematics Department, Box 951555, Los Angeles, CA 90095-1555, USA
}
\email{jmoraga@math.ucla.edu}

\author[A.~Vasilkov]{Artem Vasilkov}
\address{
National Research University Higher School of Economics, Moscow, Russian Federation.}
\email{aavasilkov@edu.hse.ru}

\subjclass[2020]{Primary 14B05, 14E30, 14L24, 14M25;
Secondary  14A20, 53D20.}
\keywords{Birational complexity, toric geometry, toric models, Fano $3$-folds}

\begin{abstract}
We prove that a general rational smooth Fano threefold admits a toric model.
More precisely, for a general rational smooth Fano threefold $X$, we show the existence of a boundary divisor $D$ for which $(X,D)\simeq_{\rm cbir} (\pp^3,H_0+H_1+H_2+H_3)$, where the $H_i$'s are
the coordinate hyperplanes.
In particular, a general rational smooth Fano threefold has birational complexity zero.
We argue that the three conditions:
rationality, generality, and smoothness
are indeed necessary for the theorem.
\end{abstract} 

\maketitle

\setcounter{tocdepth}{1} 
\tableofcontents

\section{Introduction}
Fano varieties are one of the three building blocks of smooth projective varieties.
In each dimension, there are only finitely many families of smooth Fano varieties~\cite{KMM87}.
In dimension $2$, these are known as del Pezzo surfaces and were classified in $10$ families by the Italian school of algebraic geometers.
The smooth Fano $3$-folds of Picard rank one were classified by the work of Iskovskikh~\cite{Isk77,Isk78}.
Mori and Mukai finished the classification of smooth Fano $3$-folds in $105$ families~\cite{MM81}.
The problem of rationality of smooth Fano $3$-folds
was tackled by Conte in the case of Picard rank one~\cite{Con82}.
Alzati and Bertolini settled the rationality problem for smooth Fano $3$-folds of larger Picard rank~\cite{AB92}.
To achieve this, they used criteria for rationality due to Shokurov~\cite{Sho83}.
The main theorem of this article states that a general rational smooth Fano $3$-fold satisfies a strong form of rationality, namely,
they admit a toric model (see Definition~\ref{def:toric-model}).

\begin{introthm}\label{introthm:3-fold}
Let $X$ be a rational smooth Fano threefold
that is general in its deformation family.
Then, the variety $X$ admits a toric model.
Equivalently, there exists a boundary divisor $D$ on $X$ for which
\[
(X,D) \simeq_{\rm cbir} (\pp^3,H_0+H_1+H_2+H_3),
\]
where the $H_i$'s are the coordiante hyperplanes of $\pp^3$.
\end{introthm}

The symbol $\simeq_{\rm cbir}$ stands for crepant birational isomorphism (see Definition~\ref{def:crep-bir-isom}).
Roughly speaking, it means that the underlying varieties are birationally isomorphic and this birational isomorphism respects the boundary divisor.
In Examples~\ref{ex:generality}, \ref{ex:smoothness}, and \ref{ex:rationality}, we show that the three conditions; generality, rationality, and smoothness are indeed necessary for the main theorem.
More precisely, the tools of this article allow us to prove that in {\em most} families of smooth Fano $3$-folds, all elements admit toric models.

\begin{introthm}\label{introthrm2}
Let $X$ be a rational smooth Fano threefold.
Then, $X$ admits a toric model unless 
\begin{enumerate}
\item $X$ is a special element of the family \textnumero\,10.1 of Fano $3$-folds of the form $\pp^1\times S_1$ where $S_1$ is a smooth del Pezzo surface of degree one, 
\item $X$ is a special element of the families \textnumero\,1.6, 1.8, 1.9, 1.10 of smooth Fano threefolds in the main series with anti-canonical volume 12, 16, 18, and 22, respectively, or
\item 
the general element of its deformation family is non-rational.
\end{enumerate} 
\end{introthm}

In the previous theorem, a {\em special element} means that it is contained in a Zariski closed subset of the corresponding deformation family of smooth Fano threefolds.
Special elements of the family \textnumero\,10.1 may not have anti-canonical elements with a $0$-dimensional log canonical center. This condition is necessary to admit a toric model. 
It is unclear whether special elements in the aforementioned families of smooth Fano threefolds in the main series admit anti-canonical elements with $0$-dimensional log canonical center.
As for the last case of Theorem \ref{introthrm2}, it is not known whether such Fano threefolds exist. Potentially, they could belong to families \textnumero\,1.3 and 1.5. 

\subsection{Coregularity} The coregularity of a pair $(X,D)$ is defined to be
\[
{\rm coreg}(X,D) := \dim X - \dim \mathcal{D}(X,D) -1,
\]
where $\mathcal{D}(X,D)$ is the dual complex associated to the pair (see Definition~\ref{def-dual-complex}). The coregularity of a Fano variety $X$ is defined to be
\[
{\rm coreg}(X):=\min \{{\rm coreg}(X,D) \mid
\text{$(X,D)$ is log Calabi--Yau}
\}
.
\]
The coregularity measures the singularities
of the anti-pluricanonical systems of $X$.
For instance, a $n$-dimensional Fano variety is {\em exceptional}
precisely when ${\rm coreg}(X)=\dim X$.
The coregularity of a $n$-dimensional Fano variety
is an integer in the set $\{0,\dots,n\}$.
The torus invariant boundary of a toric variety
is a $1$-complement~\footnote{A $n$-complement of a variety $X$ is a boundary divisor $D$ for which $(X,D)$ is log canonical 
and $n(K_X+D)\sim 0$.} of coregularity $0$.
Further, a Fano variety $X$ that admits a toric model
has a $1$-complement of coregularity $0$.
Conversely, a Fano variety of coregularity $0$ admits either a $1$-complement or a $2$-complement of coregularity zero (see~\cite[Theorem 4]{FFMP22}).
In~\cite[Theorem 2.1 and Proposition 2.3]{ALP23}, the authors prove that a general smooth del Pezzo surface has coregularity zero.
There are some examples of del Pezzo surfaces
of degree one with coregularity one (see~\cite[Remark 2.2]{ALP23}).
In~\cite{ALP23}, the first author, together with Avilov and Przyjalkowski, 
studied the coregularity of smooth Fano $3$-folds.
They proved that a general rational smooth Fano $3$-fold admits a $1$-complement of coregularity zero (see~\cite[Theorem 0.1]{ALP23}).

\subsection{Birational complexity}
The complexity $c(X,D)$ of a pair $(X,D)$ is defined to be 
\[
\dim(X)+\dim_\qq {\rm Cl}_\qq(X)-|D|
\]
where $|D|$ stands for the sum of the coefficients of $D$. 
The complexity measures whether a log Calabi--Yau pair is toric.
Indeed, for a log Calabi--Yau pair $(X,D)$, we have 
$c(X,D)\geq 0$ and $c(X,D)<1$ implies that the pair
$(X,\lfloor D\rfloor)$ is toric (see~\cite[Theorem 1.2]{BMSZ18}).
The {\em birational complexity} of a log Calabi--Yau pair $(X,D)$ is the infimum of the complexities
among its log Calabi--Yau crepant birational models
(see Definition~\ref{def:crep-bir-isom} and Definition~\ref{def:bir-comp})
In~\cite[Lemma 2.23]{MM24}, the second author and Mauri proved that the birational complexity is always a minimum.
A log Calabi--Yau pair $(X,D)$
admits a toric model if and only if
$c_{\rm bir}(X,D)=0$ (see, e.g.~\cite[Theorem 1.6]{MM24}).
In~\cite[Lemma 1.13]{GHK15b}, Gross, Hacking, and Keel proved that 
a log Calabi--Yau surface $(X,D)$ of index one
and coregularity zero has birational complexity zero.
Thus, by~\cite[Theorem 2.1]{ALP23} a general smooth del Pezzo surface
admits a toric model.
In~\cite[Theorem 3.3]{EFM24}, the authors show that this property holds for most Gorenstein del Pezzo surfaces.
Some of these Fano surfaces are of {\em cluster type} which is a property related to cluster algebras and is stronger than admitting a toric model
(see Question~\ref{quest:cluster-type}).
In~\cite{Ka20}, Kaloghiros gives an example 
of a log Calabi--Yau pair $(X,D)$ of dimension three, 
index one, and coregularity zero for which $X$ is birationally superrigid.
Thus, Kaloghiros' example does not admit a toric model.
In~\cite[Conjecture 1.4]{Du24}, Ducat conjectured that a $3$-dimensional rational 
log Calabi--Yau pair $(X,D)$ of index one and coregularity $0$ satisfies that $c_{\rm bir}(X,D)=0$.
The second author, together with Enwright and Figueroa, 
pose a higher-dimensional analog of this conjecture in~\cite[Conjecture 1.1]{EFM24}.
However, this conjecture is only known in dimension $2$ or in the case that $X\simeq \pp^3$ (see~\cite[Theorem 1.2]{Du24}).
Theorem~\ref{introthm:3-fold} is motivated by~\cite[Theorem 0.1]{ALP23}
and Ducat's conjecture. We are not aware
of higher-dimensional
general smooth rational Fano varieties 
that do not admit toric models (see Question~\ref{quest:toric-model}). 

\subsection{The approach}
Throughout this article, we may write $X_{\rho.m}$ for a smooth  Fano $3$-fold with ID equal $\rho.m$ in the classification of smooth Fano $3$-folds. In this notation, 
$\rho$ stands for the Picard rank, while $m$ is a positive integer indexing Fanos with such Picard rank.

The approach to the proof of Theorem~\ref{introthm:3-fold}
is a mixture of explicit geometry
of smooth Fano $3$-folds
and some more modern techniques
related to the birational complexity.
First, for any general rational smooth Fano $3$-fold $X$, we consider the $1$-complement
of coregularity zero $(X,D)$ constructed
by Avilov, Loginov, and Przyjalkowski in~\cite[Theorem 0.1]{ALP23}.
In Section~\ref{sec-blow-ups}, we show that 
whenever $Y$ is a smooth rational Fano $3$-folds that can be obtained as a blow-up $\pi\colon Y \rightarrow X$,   
the $1$-complement $(X,D)$ lifts
to a $1$-complement $(Y,D_Y)$. 
Then, using Proposition~\ref{prop:blow-up}, we show that $(Y,D_Y)$ admits a toric model
whenever $(X,D)$ does.
This reduces the problem to pairs of the form $(X,D)$, where $X$ is in the following well-known families of smooth rational Fano $3$-folds:
\begin{enumerate}
    \item[(i)] Fanos of higher index:
    \textnumero\,1.14,  1.15, 1.16, 1.17, 
    \item[(ii)] Fanos of main series: 
    \textnumero\,1.6, 1.8, 1.9, 1.10, 
    \item[(iii)] Some Fano $3$-folds with $\rho(X)>1$: 
    \textnumero\,2.18, 2.24, 2.32, 3.2.
\end{enumerate}
For any pair $(X,D)$ with $X$ as in (i)-(iii), 
we aim to either simplify the variety $X$ or the boundary $D$ by performing some ad-hoc birational modification.
In most cases, we seek a point $p\in X$ in the singular locus of $D$ or a line $\ell\subset X$ which is a log canonical center of $(X,D)$.
Then, we proceed to blow up either $p$ or $\ell$ to then 
perform a blow-down to a variety different from $X$.
In many cases, the aforementioned birational map is the projection from the point or the line in a suitable
projective space $\pp^N$ in which the Fano $3$-fold is embedded.
This leads to a sequence of crepant birational transformations: 
\[
\xymatrix{ 
(X,D)\ar@{-->}[r] & (X_1,D_1) \ar@{-->}[r] & \cdots \ar@{-->}[r] & (X_k,D_k)
}
\]
such that each $D_i$ is a $1$-complement and $X_k\simeq \pp^3$.
Thus, we can apply the main theorem of~\cite{Du24} to conclude that $(X,D)$ admits a toric model.
We close this subsection by presenting a diagram 
that summarizes the birational transformations of 
each Fano $3$-fold in (i)-(iii).
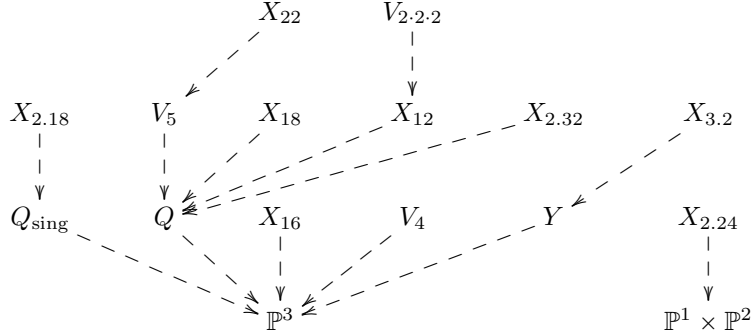
\[
\xymatrix{
 & & X_{22} \ar@{-->}[ld] & V_{2\cdot2\cdot2} \ar@{-->}[d] & & \\
X_{2.18} \ar@{-->}[d] & V_5 \ar@{-->}[d] & X_{18} \ar@{-->}[dl] & X_{12} \ar@{-->}[lld] &  X_{2.32} \ar@{-->}[llld] & X_{3.2} \ar@{-->}[ld] \\
Q_{\rm sing} \ar@{-->}[rrd] & Q \ar@{-->}[rd] & X_{16} \ar@{-->}[d] & V_4 \ar@{-->}[ld] & Y \ar@{-->}[lld] & X_{2.24} \ar@{-->}[d] \\
 & & \pp^3 & & & \pp^1\times \pp^2 \\
}
\]
In the previous diagram, $V_d$ denotes for a del Pezzo $3$-fold $X$ for which $(-K_{V_d}/2)^3=d$; 
We write $X_d$ for a smooth Fano $3$-fold in the main series with $(-K_{X_d})^3=d$; 
$Y$ stands for a special blow-down of $X_{3.2}$ (see proof of Lemma~\ref{lemma:3-2}); 
$V_{2\cdot 2 \cdot 2}$ stands for a singular intersection of three quadrics in $\pp^6$;
$Q$ stands for a smooth quadric in $\pp^4$
and $Q_{\rm sing}$ stands for a singular quadric in $\pp^4$.
Case (i) will be treated in Section~\ref{sec-higher-index}. Case (ii) will be treated in Section~\ref{sec-main-series}. Finally, Case (iii) will be treated in Section~\ref{sec-higher-picard}.

\subsection*{Acknowledgements}
The first author is supported by Russian Science Foundation under grant 24-71-10092. 
 The authors thank Alexander Kuznetsov, Yuri Prokhorov, Yuri Tschinkel, and Zhijia Zhang for helpful conversations.

\section{Preliminaries}

In what follows, all varieties are
projective and defined over $\mathbb{C}$ unless stated otherwise. We
use the language of the minimal model program (the MMP for short), see
e.g.~\cite{KM98}.

\subsection{Contractions} By a \emph{contraction} we mean a surjective
morphism $f\colon X \to Y$ of normal varieties such that $f_*\oo_X =
\oo_Y$. In particular, $f$ has connected fibers. A
\emph{fibration} is defined as a contraction $f\colon X\to Y$ such
that $\dim Y<\dim X$.

\subsection{Pairs and singularities} A \emph{pair} $(X, D)$ consists
of a normal variety $X$ and a boundary $\mathbb{Q}$-divisor $D$ with
coefficients in $[0, 1]$ such that $K_X + D$ is $\mathbb{Q}$-Cartier.
Let $\phi\colon W \to X$ be a
log resolution of $(X,D)$ and let $K_W +D_W = \phi^*(K_X +D)$.
The \emph{log discrepancy} of a prime divisor $E$ on $W$ with respect
to $(X, D)$ is $1 - \mathrm{coeff}_E D_W$ and it is denoted by $a(E, X, D)$. We
say $(X, D)$ is lc (resp. klt) if $a(E, D, B)$
is $\geq 0$ (resp.~$> 0$) for every $E$. We say that the pair is \emph{plt}, if $a(E, X, D)>0$ holds for any $\phi$-exceptional divisor $D$ and for any log resolution $\phi$. 
We say that the pair is \emph{dlt}, if $a(E, X, D)>0$ hold for any $\phi$-exceptional divisor $E$ and for some log resolution $\phi$.

\begin{definition}\label{def:crep-bir-isom}
{\em 
Let $(X,D)$ and $(Y,D_Y)$ be two pairs.
We say that these two pairs
are {\em crepant birationally equivalent} if there exists a commutative diagram
\begin{equation}
\begin{tikzcd}
& (V, D_V) \ar[rd, "\psi"] \ar[dl, swap, "\phi"] & \ \\
(X, D) \ar[rr, dashed, "\alpha"] & & (Y, D_Y)
\end{tikzcd}
\end{equation}
where $\alpha$ is a birational map, $\phi$ and $\psi$ are birational contractions, $\phi_*(D_V)=D$, $\psi_*(D_V) = D_Y$, and 
\[
K_V + D_V=\phi^*(K_X+D) = \psi^*(K_Y + D_Y).
\]
Note that here $D_V$ is not necessarily a boundary since it may have negative coefficients.  
In the previous setting, we write $(X,D)\simeq_{\rm cbir} (Y,D_Y)$.
Given $(X,D)$ and $\alpha\colon X \dashrightarrow Y$ as above, there is a unique divisor $D_Y$ for which $(X,D) \sim_{\rm cbir} (Y,D_Y)$. 
The divisor $D_Y$ is called the {\em crepant transform} of $D$ on $Y$.
}
\end{definition}

\subsection{Complements and log Calabi-Yau pairs}
\label{sect-log-CY}
Let $(X, D)$ be a log canonical pair. 
We say $(X,D)$ is \emph{log Calabi-Yau pair} (or \emph{log CY} for short) if $K_X + D
\sim_{\mathbb{Q}} 0$. In this case, $D$ is called a {\em $\mathbb{Q}$-complement} of $K_X$. If $N(K_X + D)\sim 0$ for some $N$, we say that $D$ is an {\em $N$-complement} of $K_X$. If $(X,D)$ is a log Calabi--Yau pair
and $(Y,D_Y)\sim_{\rm cbir} (X,D)$, with $D_Y\geq 0$, 
then $(Y,D_Y)$ is a log Calabi--Yau pair as well.

\subsection{Dual complex and coregularity}
Let $D=\sum D_i$ be a Cartier divisor on a smooth variety $X$. Recall that $D$ has \emph{simple normal crossings} (snc for short), if all the components $D_i$ of $D$ are smooth, and any point in $D$ has an open neighborhood in the analytic topology that is analytically equivalent to the union of coordinate hyperplanes.

\begin{definition}
\label{def-dual-complex}
{\em 
Let $D=\sum_{i=1}^r D_i$ be a simple normal crossing divisor on a smooth variety $X$.
The 
\emph{dual complex}, denoted by $\mathcal{D}(D)$ is a CW-complex constructed as follows.  
The simplices $v_Z$ of $\mathcal{D}(D)$ are in bijection with irreducible components $Z$ of the intersection $\bigcap_{i\in I} D_i$ for any subset $I\subseteq \{ 1, \ldots, r\}$, and the dimension of $v_Z$ is equal to $\#I-1$.
The gluing maps are constructed as follows. 
For any subset $I\subseteq \{ 1, \ldots, r\}$, let $Z\subset \bigcap_{i\in I} D_i$ be any irreducible component, and for any $j\in I$ let $W$ be a unique component of $\bigcap_{i\in I\setminus\{j\}} D_i$ containing $Z$. Then the gluing map is the inclusion of $v_W$ into $v_Z$ as a face of $v_Z$ that does not contain the vertex $v_j$. 
}
\end{definition}

Note that the dimension of $\mathcal{D}(D)$ does not exceed $\dim X-1$. If $\mathcal{D}(D)$ is empty, we set $\dim \mathcal{D}(D)=-1$. In what follows, for a divisor $D$ by $D^{=1}$, we denote the reduced sum of the components of $D$ with coefficient one. For an lc log CY pair $(X, D)$, we define $\mathcal{D}(X, D)$ as $\mathcal{D}(D_Y^{=1})$ where $f\colon (Y, D_Y)\to (X, D)$ is a log resolution of $(X, D)$, so that the formula
\[
K_{Y} + D_Y= f^*(K_X + D)
\]
is satisfied. It is known that the PL-homeomorphism class of $\mathcal{D}(D_Y^{=1})$ does not depend on the choice of a log resolution, see \cite[Proposition 11]{dFKX17}. 
For more results on the topology of dual complexes of Calabi-Yau pairs, see \cite{KX16}. 

\begin{definition}[{\cite[7.9]{Sho00}},{\cite{Mor22}}]
\label{defin-regularity}
{\em 
Let $X$ be a klt Fano variety of dimension $n$. By the \emph{regularity} $\mathrm{reg}(X, D)$ of an lc log CY pair $(X, D)$ we mean the number $\dim \mathcal{D}(X, D)$. For $l\geq 1$, we define the {\em $l$-th regularity} of $X$ by the formula 
\[
\mathrm{reg}_l(X) = \max \left\{ \mathrm{reg}(X, D)\ |\ D\in \frac{1}{l} |-lK_X| \right\}.
\]
Then the \emph{regularity} of $X$ is 
\[
\mathrm{reg}(X) = \max_{l\geq 1} \{\mathrm{reg}_l(X)\}.
\]
Note that $\mathrm{reg}(X)\in \{-1, 0,\ldots, \dim X-1\}$ where by convention we say that the dimension of the empty set is $-1$. The \emph{coregularity} of an lc log CY pair $(X, D)$ is defined as the number $n-1-\mathrm{reg}(X, D)$. Also, we define:
\[
\mathrm{coreg}_l(X) = n-1-\mathrm{reg}_l(X),\quad \text{ and } \quad \quad \mathrm{coreg}(X) = n-1-\mathrm{reg}(X).
\]
Clearly, one has $\mathrm{reg}_l(X)\leq \mathrm{reg}_{kl} (X)$ for any $k,l\geq1$.
}
\end{definition}

\subsection{Toric models} In this subsection, we recall the concepts of complexity, birational complexity, and toric models for log pairs.

\begin{definition}\label{def:comp}
{\em 
Let $(X,D)$ be a log pair. We define the {\em complexity} of $(X,D)$, denoted by $c(X,D)$ to be $\dim X +\dim_\qq {\rm Cl}_\qq(X) -|D|$, where $|D|$ stands for the sum of the coefficients of $D$.
}
\end{definition}

\begin{definition}\label{def:bir-comp}
{\em 
The {\em birational complexity} of a log Calabi--Yau pair $(X,D)$ is defined to be:
\[
c_{\rm bir}(X,D):=\inf\{ 
c(X',D') \mid (X',D') \simeq_{\rm cbir}(X,D)
\text{ and $D'\ge 0$}
\}.
\]
The {\em birational complexity} of a Fano variety $X$, on the other hand, is defined to be:
\[
c_{\rm bir}(X):=\inf\{ c_{\rm bir} (X,D) \mid 
\text{
$(X,D)$ is log Calabi--Yau
}
\}.
\]
}
\end{definition}

\begin{definition}\label{def:toric-model}
{\em 
We say that a log CY pair $(X, B)$ has a {\em toric model} if 
the isomorphism holds
\[
(X,D) \simeq_{\rm bir} (\pp^n,H_0+\dots+H_n),
\]
for $n=\dim X$. 
We say that a Fano variety has a {\em toric model} if it admits a $1$-complement $D$ such that $(X,D)$ has a toric model. 
Note that if a Fano variety $X$ has a toric model then it must be rational. It is worth noting that any toric pair $(X, D)$ where $D$ is the reduced sum of the torus invariant divisors is crepant birational equivalent to $(\pp^n,H_0+\dots+H_n)$.
}
\end{definition} 

\subsection{Smooth toric Fano threefolds}
In this subsection, we recall the list of smooth toric Fano threefolds.
\begin{proposition}[{\cite{IP99}}]

\label{prop-toric-fanos}
Let $X_{\rm \rho.m}$ be a smooth Fano $3$-fold. 
Then, the variety $X_{\rm \rho.m}$ is toric if and only if
\[
\rho.m \in 
\{\,1.17,\ 2.33,\ 2.34,\ 2.35,\ 2.36,\ 3.25,\ 3.26,\ 3.27,\ 3.28,\ 3.29,\
3.30,\ 3.31,\ 4.9,\ 4.10,\ 4.11,\ 4.12,\ 5.2,\ 5.3
\}.
\]
In particular, any such $X_{\rm \rho.m}$ admits a toric model.
\end{proposition}

\section{Fano threefolds of higher index}

\label{sec-higher-index}

Recall that the {\em index} $i(X)$ of a Fano variety $X$ is a maximal natural number $m$ such that $-K_X\sim mH$ where $H$ is a Cartier divisor.
It is well-known that if $X$ is a smooth Fano threefold then $i(X)\leq 4$. Moreover, if $i(X)=4$ then $X\simeq \mathbb{P}^3$, and if $i(X)=3$ then $X$ is isomorphic to a quadric in $\mathbb{P}^4$. In this section, we treat the case of Fano threefolds of Picard rank $1$ and index at least $2$.

\begin{theorem}[{\cite[Theorem 1.2]{Du24}}]
\label{thm-Ducat}
\label{thm-projective-space-toric-model}
Let $(\mathbb{P}^3, D)$ be an lc log CY pair where $D=\sum D_i$ is a reduced divisor. If $\mathrm{coreg}(\mathbb{P}^3, D)=~0$, then $(\mathbb{P}^3, D)$ admits a toric model. 
\end{theorem}

\begin{lemma}
\label{lem-singular-boundary}
Let $(X, D=\sum D_i)$ be an lc log CY pair where $X$ is a smooth Fano threefold. Assume that $\mathrm{coreg}(X,D)\leq 1$. Then there is a point $x\in X$ such that $D$ is singular in $x$.
\end{lemma}
\begin{proof}
Assume that $D$ is smooth. Then $D$ is irreducible (recall that $D$ is ample so its support is connected). Then the pair $(X, D)$ is log smooth, and its dual complex is a point, so $\mathrm{coreg}(X, D)=2$. This contradiction shows that $D$ is singular. 
\end{proof}

\begin{lemma}
\label{lem-quadric-cluster-type}
Let $(Q, D=\sum D_i)$ be an lc log CY pair where $Q\subset \mathbb{P}^4$ is a smooth quadric and $\mathrm{coreg}(Q, D)\leq 1$. Then $(Q, D)$ is crepant birational to $(\mathbb{P}^3, D'=\sum D'_i)$. In particular, if $\mathrm{coreg}(Q, D)=0$ then $(Q, D)$ admits a toric model. Consequently, $Q$ admits a toric model.
\end{lemma}
\begin{proof}
By Lemma \ref{lem-singular-boundary}, there exists a point $Q\in X$ which is a singular point of $D$.  
Now, project $Q$ from the point $P$ to $\mathbb{P}^3$ to obtain a pair $(\mathbb{P}^3, D'=\sum D'_i)$ which is crepant equivalent to $(Q, D)$. Apply Theorem \ref{thm-projective-space-toric-model} to conclude that if $\mathrm{coreg}(Q, D)=0$ then $(Q, D)$ admits a toric model. Further, there exists a boundary $D$ on $Q$ with $\mathrm{coreg}(Q, D)=0$, cf. \cite[Lemma 5.3]{ALP23}. This concludes the proof.
\end{proof}

If $i(X)=2$ then $X$ is called a \emph{del Pezzo threefold}. Put $d=(-K_X/2)^3$. Then $1\leq d \leq 5$. In this section, we assume that $\rho(X)=1$. We denote the corresponding Fano threefold by $V_d$. It is known that $V_1$, $V_2$ and $V_3$ are non-rational, so we have to consider the varieties $V_4$ and $V_5$. By a \emph{line} on $X$ we mean a smooth rational curve with $(-K_X/2)\cdot L=1$. 
The following lemma is well-known.

\begin{lemma}
\label{lem-dP-line}
Let $X=V_d$ be a smooth del Pezzo threefold with $d\geq 3$ embedded into projective space via the linear system $|-K_X/2|$. Then through any point $x\in X$ passes a line that belongs to $X$.
\end{lemma}
\begin{proof}
By \cite[Proposition 2.2.8]{KPrSh18} the Hilbert scheme of lines $F_1(X)$ on $X$ is a smooth surface. Let $\mathcal{L}(X)\to F_1(X)$ be the universal family of lines on $X$. By \cite[Proposition 2.2.6]{KPrSh18}, the natural map $ev_{\mathcal{L}(X)}\colon \mathcal{L}(X)\to X$ is surjective, and the claim follows.
\end{proof}

\begin{lemma}
\label{lem-dP-line2}
Let $V_d$ be a smooth del Pezzo threefold with $d\geq 3$ embedded into projective space via the linear system $|-K_X/2|$. 
Let $(V_d, D=\sum D_i)$ be an lc log CY pair. 
Assume that $D$ is reducible. Then there exists a line $L$ on $V_d$ such that $L\subset D$.
\end{lemma}
\begin{proof}
Since a boundary $D$ is reducible, we have $D=D_1+D_2$. Indeed, note that $D\sim -K_X \sim 2H$ where~$H$ is a hyperplane section, and $\rho(X)=1$. 
If each $D_i$ is smooth, then each $D_i$ is an anti-canonically embedded degree $d$ del Pezzo surface. By assumption, we have $3\leq d\leq 5$. It is classically known that such surfaces contain lines, so we can find $L\subset D_i$. Assume that some $D_i$ is singular, so there exists a point $P\in D_i$ with $\mathrm{mult}_P D_i\geq 2$. Consider a line $L$ on $X$ that passes through $P$, which exists by Lemma \ref{lem-dP-line}. Since $D_i\sim H$, we have $L\cdot D_i=1$. On the other hand, if $L\not\subset D_i$ we have $L\cdot D_i\geq \mathrm{mult}_P D_i\geq 2$, which is a contradiction. So we have $L\subset D$. 
\end{proof}

For an irreducible boundary $D$ in $V_d$, we do not know whether there exists a line $L$ with the property $L\subset D$. Note that the assumption that $D$ is reducible implies $\mathrm{coreg}(V_d, D)\leq 1$.

\begin{lemma}
\label{V_4-model}
Let $(V_4, D)$ be an lc log CY pair. Assume that $D$ is reducible. Then $(V_4, D)$ is crepant equivalent to $(\mathbb{P}^3, D'=\sum D'_i)$. If $\mathrm{coreg}(V_4, D)=0$, then $(V_4, D)$ admits a toric model. Consequently, $V_4$ admits a toric model. 
\end{lemma}
\begin{proof}
It is known that $X$ can be embedded in $\mathbb{P}^5$ where it is realized as the intersection of two quadric hypersurfaces. 
By Lemma \ref{lem-dP-line2}, there exists a line $L$ on $V_4$ such that $L\subset D$. 
Consider the projection of $V_4$ from $L$ to $\mathbb{P}^3$.  
We obtain the following commutative diagram:
\[
\begin{tikzcd}
Y \ar[rrd, "\psi"] \ar[d, "\phi"] &  & \ \\
V_4 \subset
\mathbb{P}^{5} \ar[rr, dashed, "\alpha"] & & \mathbb{P}^3
\end{tikzcd}
\]
where $\phi$ is the blow-up of $L$ and $\psi$ is a contraction of a ruled surface to a curve of genus $2$ in $\mathbb{P}^3$. 
Since $L\subset D$, we see that for the log pullback $(Y, D_Y)$ of $(V_4, D)$, the divisor $D_Y$ is a boundary. Hence in the pushdown $(\mathbb{P}^3, D')$, $D'$ is boundary as well. 
In particular, if $\mathrm{coreg}(V_4, D)=0$, then $\mathrm{coreg}(\mathbb{P}^3, D')=0$. By Theorem \ref{thm-Ducat} we have that $(\mathbb{P}^3, D')$ admits a toric model, hence $(V_4, D)$ admits a toric model as well. 
Further, there exists a boundary $D=D_1+D_2$ on $V_4$ with $\mathrm{coreg}(V_4, D)=0$  
as constructed in ~\cite[Lemma 5.4]{ALP23}. This concludes the proof.
\end{proof}

\begin{lemma}
\label{V_5-model}
Let $(V_5, D)$ be an lc log CY pair. Assume that $D$ is reducible. Then $(V_5, D)$ is crepant equivalent to $(\mathbb{P}^3, D'=\sum D'_i)$. If $\mathrm{coreg}(V_5, D)=0$, then $(V_5, D)$ admits a toric model. Consequently, $V_5$ admits a toric model. 
\end{lemma}
\begin{proof}
It is known that $X$ can be realized as the intersection of the Grassmannian $\mathrm{Gr}(2,5)$ in $\mathbb{P}^9$ with a subspace $\mathbb{P}^6$. So we have $X\subset\mathbb{P}^6$. 
By Lemma \ref{lem-dP-line2}, there exists a line $L$ on $V_5$ such that $L\subset D$. 
Consider the projection of $V_5$ from $L$ to $\mathbb{P}^4$.  
We obtain the following commutative diagram:
\[
\begin{tikzcd}
Y \ar[rrd, "\psi"] \ar[d, "\phi"] &  & \ \\
V_5 \subset
\mathbb{P}^{6} \ar[rr, dashed, "\alpha"] & & Q\subset\mathbb{P}^4
\end{tikzcd}
\]
where $\phi$ is the blow-up of $L$ and $\psi$ is a contraction of a ruled surface to a rational curve in $Q$. Since $L\subset D$, we see that for the log pullback $(Y, D_Y)$ of $(V_5, D)$, the divisor $D_Y$ is a boundary. Hence in the pushdown $(\mathbb{P}^3, D')$, $D'$ is boundary as well. 
In particular, if $\mathrm{coreg}(V_5, D)=0$, then $\mathrm{coreg}(\mathbb{P}^3, D')=0$. By Theorem \ref{thm-Ducat} we have that $(\mathbb{P}^3, D')$ admits a toric model, hence $(V_5, D)$ admits a toric model as well. 
Further, there exists a boundary $D=D_1+D_2$ on $V_5$ with $\mathrm{coreg}(V_5, D)=0$  
as constructed in ~\cite[Lemma 5.4]{ALP23}. This concludes the proof.
\end{proof}

\section{Fano threefolds of main series}
\label{sec-main-series}
In this section, we work in the following setting. Let $X$ be a smooth Fano threefold with $\rho(X)=1$ and $i(X)=1$. In this case, $X$ is called a Fano threefold of the main series. In what follows, by the \emph{genus} of a smooth Fano threefold $X$ of the main series, we mean the number:
\[
g(X) = h^0(X, \oo(-K_X))-2.
\] 
It is known that $2\leq g\leq 12$, $g\neq 11$. In what follows, we denote by $X_d$ a Fano threefold of the main series with $(-K_X)^3=d$. It is known that $d=2g(X)-2$.

\subsection{Lines and conics}
By a {\em line} on a Fano threefold $X$ of the main series we mean a smooth rational curve $L$ such that $L\cdot (-K_X)=1$. By \cite[Proposition 1 and Corollary 1]{Isk89}, we have either 
\[
N_{L/X}=\oo\oplus\oo(-1) \quad \quad \text{or} \quad \quad N_{L/X}=\oo(-1)\oplus\oo(2).
\] 
If for all lines on $X$ the second possibility is realized, then $X$ is called \emph{exotic}. It is known that for $g(X)\geq 9$ there exists a unique exotic Fano threefold $X$, namely, the Mukai-Umemura example \cite{MU82}. 
It is also known that $F_1(X)$ is one-dimensional projective variety, and smooth points of $F_1(X)$ correspond to a lines whose normal bundle is $\oo\oplus\oo(-1)$, see \cite[Corollary 2.1.6]{KPrSh18}. For a general $X$ we have that $F_1(X)$ is a smooth curve, so all lines have normal bundle $\oo\oplus\oo(-1)$ \cite[Theorem 4.2.7]{IP99}. Also, through any point in $X$ passes at most finitely many lines. 
It follows that for a general smooth Fano threefold of the main series, we always can find a line with normal bundle of the form $\oo\oplus\oo(-1)$. 

By a {\em conic} on a Fano threefold $X$ of the main series we mean a smooth rational curve $L$ such that $L\cdot (-K_X)=2$. Denote by $F_2(X)$ the Hilbert schemes on conics on $X$. 
It is also known that if $g(X)\geq 7$, then  $F_2(X)$ is a  smooth irreducible surface \cite[Proposition 2.3.6]{KPrSh18}. Moreover, a general conic on $X$ intersects only finitely many lines on $X$. 

\begin{lemma}
\label{lem-x16-toric-model}
A general Fano variety $X_{16}$ admits a toric model. 
\end{lemma}
\begin{proof}
Let $X$ be a general Fano threefold in the family \textnumero\,1.8. Then $g(X)=9$ and $X\subset \mathbb{P}^{10}$. We may assume that it is not exotic. By
\cite[Theorem 4.3.7(iii)]{IP99} the variety $X$ fits in the following commutative diagram:
\begin{equation}
\begin{tikzcd}
Y \ar[rr, dashed, "\alpha"] \ar[d, swap, "\phi"] &  & Y' = \mathrm{Bl}_\Gamma\mathbb{P}^3 \ar[d, "\psi"]  \\
X = X_{16} \subset
\mathbb{P}^{10} \ar[rr, dashed, "\beta"] & & \mathbb{P}^3 
\end{tikzcd}
\end{equation}
where 
\begin{itemize}
\item
$\phi$ is the blow-up of a line on $X$, 
\item
$\alpha$ is a composition of flops, 
\item
a birational map $\beta$ is given by the linear system $|H-2L|$ where $H$ is a hyperplane section, 
\item 
$\psi$ is the blow-up of a smooth non-hyperelliptic curve $\Gamma\subset\mathbb{P}^3$ of degree $7$ and genus $3$. 
\end{itemize}
By \cite[Lemma 6.4]{ALP23} there exists a reduced irreducible boundary $D$ on $X_{16}$ with $L\subset D$ and $\mathrm{coreg}(X_{16}, D)=0$. Thus, its crepant transform $D'$ on $\mathbb{P}^3$ is also a boundary. By Theorem~\ref{thm-Ducat}, we have that $(\mathbb{P}^3, D')$ admits a toric model, hence $(X_{16}, D)$ admits a toric model as well.
\end{proof}

\begin{lemma}
\label{lem-x18-toric-model}
A general Fano variety $X_{18}$ admits a toric model. 
\end{lemma}

\begin{proof}
Let $X$ be a general Fano threefold in the family \textnumero\,1.9. Then $g(X)=10$ and $X\subset \mathbb{P}^{11}$. By
\cite[Theorem 4.3.7(ii)]{IP99} the variety $X$ fits in the following commutative diagram 
\begin{equation}
\begin{tikzcd}
Y \ar[rr, dashed, "\alpha"] \ar[d, swap, "\phi"] &  & Y' = \mathrm{Bl}_\Gamma Q \ar[d, "\psi"]  \\
X = X_{18} \subset
\mathbb{P}^{11} \ar[rr, dashed, "\beta"] & & Q\subset \mathbb{P}^4
\end{tikzcd}
\end{equation}
where 
\begin{itemize}
\item
$\phi$ is the blow-up of a line on $X$, 
\item
$\alpha$ is a composition of flops, 
\item
a birational map $\beta$ is given by the linear system $|H-2L|$ where $H$ is a hyperplane section, 
\item 
$\psi$ is the blow-up of a smooth curve $\Gamma\subset Q$ of degree $7$ and genus $2$ on  a smooth quadric $Q\subset\mathbb{P}^4$. 
\end{itemize}
By \cite[Lemma 6.5]{ALP23} there exists a reduced irreducibly boundary $D$ on $X_{18}$ with $L\subset D$ and $\mathrm{coreg}(X_{18}, D)=0$. Thus, its crepant transform $D'$ on $Q$ is also a boundary. By Lemma \ref{lem-quadric-cluster-type} we have that $(Q, D')$ admits a toric model, hence $(X_{18}, D)$ admits a toric model as well.
\end{proof}

\begin{lemma}
\label{lem-x22-toric-model}
A general Fano variety $X_{22}$ admits a toric model. 
\end{lemma}
\begin{proof}
Let $X$ be a Fano threefold in the family \textnumero\,1.10 different from the Mukai-Umemura example, so it is not exotic. Then $g(X)=12$ and $X\subset \mathbb{P}^{13}$. By
\cite[Theorem 4.3.7(i)]{IP99} the variety $X$ fits in the following commutative diagram 
\begin{equation}
\begin{tikzcd}
Y \ar[rr, dashed, "\alpha"] \ar[d, swap, "\phi"] &  & Y' = \mathrm{Bl}_\Gamma V_5 \ar[d, "\psi"]  \\
X = X_{22} \subset
\mathbb{P}^{13} \ar[rr, dashed, "\beta"] & & V_5\subset \mathbb{P}^6
\end{tikzcd}
\end{equation}
where 
\begin{itemize}
\item
$\phi$ is the blow-up of a line on $X$, 
\item
$\alpha$ is a composition of flops, 
\item
a birational map $\beta$ is given by the linear system $|H-2L|$ where $H$ is a hyperplane section, 
\item 
$\psi$ is blow-up of a smooth rational curve $\Gamma\subset V_5$ of degree $5$ where $V_5$ is a del Pezzo threefold of degree $5$. 
\end{itemize}
By \cite[Lemma 6.6]{ALP23} there exists a reduced irreducible boundary $D$ on $X_{22}$ with $L\subset D$ and $\mathrm{coreg}(X_{22}, D)=0$. Thus, its crepant transform $D'$ on $V_5$ is also a boundary. Moreover, $D'$ has the form $D'=D'_1+D'_2$ where $D'_1$ and $D'_2$ are hyperplane sections, so $D'$ is reducible. 
By Lemma \ref{V_5-model} it follows that $(V_5, D')$ admits a toric model, hence $(X_{22}, D)$ admits a toric model as well.
\end{proof}

Now we treat the case of Fano threefolds in the family \textnumero\,1.6. 

\begin{lemma}
\label{lem-1.6-to-222}
For a general $X_{12}$, there exists a boundary $D$ on $X$ such that $D$ contains a smooth conic $C$ on $X$, and such that $\mathrm{coreg}(X, D)=0$.
\end{lemma}
\begin{proof}
To prove the claim, we recall a classical construction due to Iskovskikh.
Let $X$ be a general Fano threefold in the family \textnumero1.6. Then $g(X)=7$ and $X\subset \mathbb{P}^8$. By \cite[Theorem 6.1(vii)]{Isk78},  
there exists the following commutative diagram 
\begin{equation}
\label{diag-X12}
\begin{tikzcd}
Y \ar[rrd, "\alpha"] \ar[d, swap, "\phi"] &  & \ \\
X = X_{12} \subset
\mathbb{P}^{8} \ar[rr, dashed, "\beta"] & & V_{2\cdot 2\cdot 2} \subset \mathbb{P}^6
\end{tikzcd}
\end{equation}
where 
\begin{itemize}
\item
$\phi$ is the blow-up of a line $L\subset X$, 
\item
$\alpha$ is a small contraction, and it contracts precisely the strict transforms of $8$ lines on $X$ passing through $L$,
\item
a birational map $\beta$ is given by the linear system $|H-L|$, 
\item 
the  morphism $\alpha$ is given by the linear system $|H-E|$ where $E\simeq \mathbb{F}_1$ is the $\phi$-exceptional divisor. 
\end{itemize}
The image of $\alpha$ is the intersection of $3$ quadrics $V=V_{2\cdot 2\cdot 2}$ which has $8$ ordinary double points $p_1,\ldots, p_8$ as singularities. Moreover, $V$ contains the isomorphic image $F=F_3$ of $E$, and $F$ is a surface of degree $3$ in $\mathbb{P}^4\subset \mathbb{P}^6$ that passes through $p_1,\ldots, p_8$. 
Consider a pencil of hyperplane sections $\mathcal{H}$ on $V$ that contain $F$. Elements of $\mathcal{H}$ on $V$ are reducible. Each such element is the union of $F$ and a del Pezzo surface $R$ of degree $5$. Let $\mathcal{R}$ be the pencil of such residual surfaces. One check that $R\in \mathcal {R}$ restricts to $F$ as an element of the linear system $|-K_{F}-\sum p_i|$ on $F\simeq \mathbb{F}_1$. 

Consider the pair $(V, D_V)$ where $D_V = F + R$ and $R\in \mathcal{R}$ is such an element that $R|_F$ is an irreducible curve with one node $p$. Such a curve exists provided that $X$ is general, as shown by the equations in the proof of \cite[Lemma 6.2]{ALP23}. Also, this node is disjoint from $p_1,\ldots p_8$. Note that $R$ may be singular, but it has no worse than lc singularities. Then the pair $(V, D_V)$ is lc by inversion of adjunction on $F$. Let $(Y, D_Y)$ be its strict transform on $Y$. By \cite[Lemma 4.3]{ALP23} applied locally near $p$, we see that $\dim \mathcal{D}(Y, D_Y)=2$ and so  $\mathrm{coreg}(Y, D_Y)=\mathrm{coreg}(X, D)=0$.

We show that the surface $R$ contains a line $L'$. Indeed, a general element in $\mathcal{R}$ is a quintic del Pezzo surface, so it contains a line in $\mathbb{P}^8$. But a line can degenerate only to a line, so there exists a line $L'\subset R$. Then, $L'$ intersects $F$ at one point, so the strict transform of $L'$ on $X$ is a conic $C$. Also, we have $C\subset D$, and the claim follows.
\end{proof}

\begin{lemma}
\label{lem-x12-toric-model}
A general Fano variety $X_{12}$ admits a toric model.
\end{lemma}

\begin{proof}
Put $X=X_{12}$. 
By Lemma \ref{lem-1.6-to-222}, there exists a boundary $D$ on $X$ such that $D$ contains a smooth conic~$C$ on $X$, and such that $\mathrm{coreg}(X, D)=0$. 
By \cite[Theorem 5.9]{KPr23}, for any smooth conic $C$ on $X$, there exists the following commutative diagram\footnote{Note that in the statement of \cite[Theorem 5.9]{KPr23}, the conic $C$ is assumed to be singular, but the proof works for a smooth conic as well.}:
\begin{equation}
\begin{tikzcd}
Y \ar[rr, dashed, "\alpha"] \ar[d, "\phi"] &  & Y' = \mathrm{Bl}_\Gamma Q \ar[d, "\psi"]  \\
X = X_{12} \subset
\mathbb{P}^{8} \ar[rr, dashed, "\beta"] & & Q\subset \mathbb{P}^4 
\end{tikzcd}
\end{equation}
where 
\begin{itemize}
    \item 
    $\phi$ is the blow-up of $C$ on $X$,
    \item 
    $\alpha$ is a flop,
    \item 
    $\psi$ is the blow-up of a curve $\Gamma$ of degree $10$ and genus $7$ on $Q$,
    \item 
    a birational map $\beta$ is given by the linear system $|H-C|$.
\end{itemize}
Since $C\subset D$, the strict transform $D'$ of $D$ on $Q$ will be a boundary. Also, $(Q, D')$ is an lc log CY pair of coregularity~$0$. Applying Lemma \ref{lem-quadric-cluster-type}, we conclude that $(Q, D')$ admits a toric model, hence $(X, D)$ admits a toric model as well.   
\end{proof}

\section{\texorpdfstring{Fano threefolds with $\rho(X)>1$}{Fano threefolds with rho>1}}
\label{sec-some-fano-threefolds}

In this section, we aim to prove the main theorem for smooth Fano $3$-folds $X$ with $\rho(X)>1$
which are not blow-ups of other smooth Fano $3$-folds.

\label{sec-higher-picard}

\begin{lemma}
\label{lem-2.32-Q}
A Fano threefold in the family \textnumero\,2.32 admits a toric model. 
\end{lemma}

\begin{proof}
Let $X$ be a Fano threefold in the family \textnumero\,2.32, that is, $X$ is a divisor of bidegree $(1,1)$ in $\mathbb{P}^2\times\mathbb{P}^2$. We can blow up a smooth rational curve $C\subset X$ which is of type $(1,1)$ in $\mathbb{P}^2\times\mathbb{P}^2$ to obtain a Fano variety~$Y$ in the family \textnumero\,3.20. It fits into the following commutative diagram: 
\begin{equation}
\label{diagram-2-32}
\begin{tikzcd}
Y \ar[rrd, "\psi"] \ar[d, "\phi"] &  & \ \\
X \subset
\mathbb{P}^{2}\times\mathbb{P}^2 \ar[rr, dashed, "\alpha"] & & Q\subset \mathbb{P}^4.
\end{tikzcd}
\end{equation}
Here the map $\psi$ is the blow-up of two disjoint lines $L_1$ and $L_2$ on a smooth quadric $Q\subset \mathbb{P}^4$.
By \cite[Lemma 11.3]{ALP23} we can pick a boundary $D=D_1+D_2+D_3+D_4$ on $X$ such that $D_1, D_2\sim (1,0)$, $D_3,D_4\sim (0,1)$, and $\mathrm{coreg}(X, D)=0$. Moreover, we can pick $D$ such that $C\subset D$. Then in its log pullback $(Y, D_Y)$, $D_Y$ is a boundary, and hence in the pushdown $(Q, D')$ we have that $D'$ is a boundary as well. By Lemma \ref{lem-quadric-cluster-type} we see that $(Q, D')$ admits a toric model, hence $(X, D)$ admits a toric model as well.
\end{proof}

\begin{lemma}\label{lem:2.24}
A Fano threefold in the family \textnumero\,2.24 admits a toric model. 
\end{lemma}
\begin{proof}
Let $X$ be a Fano threefold of type \textnumero\,2.24, that is, $X$ is a divisor of bidegree $(1,2)$ in $\mathbb{P}^2\times\mathbb{P}^2$. We establish the desired result by working explicitly in coordinates. By definition, $X$ is given by equation 
\[
x_0f_0(y_0,y_1,y_2)+x_1f_1(y_0,y_1,y_2)+x_2f_2(y_0,y_1,y_2)=0
\]
where $\deg f_i(y_0,y_1,y_2)=2$.  
Let $C$ be a degenerate fiber of a natural conic bundle $\mathrm{pr}_1\colon X\to\mathbb{P}^2$. After a change of coordinates we may assume that $C$ lies over a point $P=[1: 0: 0]$ and coincides with the set $\{P\}\times \{f_0(y)=y_1y_2=0\}$. The induced map $\gamma\colon \text{Bl}_{P}\mathbb{P}^2\times\mathbb{P}^2\rightarrow\mathbb{P}^1\times\mathbb{P}^2$ regularizes the projection map $\alpha\colon \mathbb{P}^2\times\mathbb{P}^2\dashrightarrow\mathbb{P}^1\times\mathbb{P}^2$ from the plane $\{P\}\times\mathbb{P}^2$. Let $\phi\colon Y\rightarrow X$ be the induced blow-up of $X$ in the curve $C$. We obtain the following commutative diagram:
\begin{equation}
\begin{tikzcd}
 Y  \ar[d, "\phi"] \ar[rrd, "\psi"]  & & \ \\
X\subset
\mathbb{P}^{2}\times\mathbb{P}^2 \ar[rr, dashed, "\alpha"] & &\mathbb{P}^1\times\mathbb{P}^2.
\end{tikzcd}
\end{equation}
Note that the map $\alpha$ is birational on $X$. Indeed, using the equation for $X$, we can express $x_0$ in terms of other variables, hence we  obtain the inverse map  $\mathbb{P}^1\times\mathbb{P}^2\dashrightarrow X$, defined on the set $\{f_0\neq0\}\subset \mathbb{P}^1\times\mathbb{P}^2$.

Let $D_1=\{x_1=0\}$, $D_2=\{x_2=0\}$ and $D_3=\{y_2=0\}$ be divisors on $X$. Put $D=D_1+D_2+D_3$, then $K_X+D\sim 0$. Denote by $D_Y$ the log pullback of $D$ on $Y$ under the map $\phi$ and let $D'$ be its pushdown under the map $\psi$. Note that the exceptional divisor $E$ of the morphism $\gamma$ maps isomorphically on $\mathbb{P}^1\times\mathbb{P}^2$ under the map $\psi$. We denote by $E|_{Y}=E_1+E_2$ the exceptional divisor of $\phi$. By the construction of the blow-up, $E_1+E_2$ maps to $\{f_0(y)=0\}\subset\mathbb{P}^1\times\mathbb{P}^2$ under the map $\psi$. Let us note that $\alpha_* {D_i}$ is given by the same equation in $\mathbb{P}^1\times\mathbb{P}^2$ as $D_i$ in $\mathbb{P}^2\times\mathbb{P}^2$. Consequently, we have $D'=B_0+B_1+B_2+H_0+H_1$, where $B_i=\{y_i=0\}$ and $H_i=\{x_i=0\}$. We can see that $D'$ is invariant under standard torus action on $\mathbb{P}^1\times\mathbb{P}^2$, hence $(\mathbb{P}^1\times\mathbb{P}^2,D')$ is a toric pair. We conclude that $(X,D)$ admits a toric model. 

\end{proof}

\begin{lemma}[{cf. \cite[Proposition 5.1]{CTT24}}]
\label{lem-2.18}
A Fano threefold in the family \textnumero\,2.18 admits a toric model. 
\end{lemma}
\begin{proof}
Let $X$ be a Fano threefold of type \textnumero\,2.18, that is, $X$ is double cover of $\mathbb{P}^1\times\mathbb{P}^2$ ramified in a divisor $R$ of bidegree $(2,2)$. Note that the natural map $\pi_1\colon X\to \mathbb{P}^1$ is a quadric bundle. Also, we have that $\pi_2\colon X\to \mathbb{P}^2$ is a conic bundle whose discriminant curve $\Delta$ is a quartic on $\mathbb{P}^2$. Let $H$ be the pullback of a line from $\mathbb{P}^2$ and let $F$ be the pullback of a point from $\mathbb{P}^1$. We have 
$
K_X \sim -F - 2H.
$
Using \cite[Lemma 11.5]{ALP23}, we can pick a boundary $D=D_1+D_2+D_3$ where $D_1\sim H$, $D_2\sim H$, $D_3\sim F$ with $K_X+D\sim0$ and such that $\mathrm{coreg}(X, D)=0$.

Let $L=L_1+L_2$ be a fiber of $\pi_2$ over a general point of $\Delta$, so $L_1$ and $L_2$ are smooth rational curves intersecting in one point. We may also assume that $L\subset D_1$. Observe that the divisor $H+F$ is very ample and the linear system $|H+F|$ defines an embedding $g\colon X\to \mathbb{P}^6$. In fact, the image $g(X)$ belongs to the cone over $\mathbb{P}^1\times\mathbb{P}^2$ which is embedded in $\mathbb{P}^5$. Note that $g(L_1)$ is a line in $\mathbb{P}^6$. 

Let $\phi\colon Y\to X$ be the blow-up of $L_1$ on $X$. Let $E\simeq \mathbb{F}_1$ be the exceptional divisor of $f$. Then $Y$ is a weak Fano threefold. Put $B=\phi^*(H+F) - E$. Then $B^3=2$ and the divisor $B$ is big and nef. On $Y$, the map given by the linear system $|B|$ regularizes the projection map from $L_1$ on $X\subset \mathbb{P}^6$. So we have a small birational morphism $\psi\colon Y\to \mathbb{P}^4$ which contracts the strict preimage of $L_2$ on $X_1$. We obtain the following diagram:
\begin{equation}
\begin{tikzcd}
Y \ar[rrd, "\psi"] \ar[d, "\phi" ] &  & \ \\
X \subset
\mathbb{P}^{6} \ar[rr, dashed, "\alpha"] & & Q\subset \mathbb{P}^4.
\end{tikzcd}
\end{equation}
Note that the image $Q=\psi(Y)$ has degree $2$ in $\mathbb{P}^4$, hence $Q$ is a singular quadric hypersurface. In fact, it is a quadric of rank $4$. 
Let $D_Q$ be the crepant transform of $D$ on $Q$. Project $Q$ to $\mathbb{P}^3$ from a point that belongs to $D_Q$ to obtain a pair $(\mathbb{P}^3, D')$ where $D'$ is the strict transform of $D_Q$ on $\mathbb{P}^3$. By Theorem \ref{thm-Ducat}, the pair $(\mathbb{P}^3, D')$ admits a toric model. Hence the pair $(X, D)$ admits a toric model as well.
\end{proof}

\begin{lemma}
\label{lemma:3-2}
Let $X$ be a Fano threefold in the family \textnumero \,$3.2$.
Then $X$ admits a toric model. 
\end{lemma}

\begin{proof}
We have that $X$ is a smooth divisor in a $\mathbb{P}^2$-bundle over
$\mathbb{P}^1\times\mathbb{P}^1$ of the form
\[
\mathbb{P}=\mathbb{P}_{\mathbb{P}^1\times\mathbb{P}^1}(\oo^{\oplus 2}\oplus\oo(-1,-1)),\quad \quad \quad X\sim 2M + F_1 \sim 2L + 3F_1 + 2F_2  
\]
where $M$ is the tautological divisor, $F_1$ is the pullback of the divisor $l_1$ of bidegree $(1,0)$ on $\mathbb{P}^1\times\mathbb{P}^1$, $F_2$ is the pullback of the divisor $l_2$ of bidegree $(0,1)$ on $\mathbb{P}^1\times\mathbb{P}^1$, and $L\sim M-F_1-F_2$. Note that a general fiber $F_1$ is a smooth del Pezzo surface of degree $6$ and $F_2$ is a smooth del Pezzo surface of degree $3$.
By \cite[Lemma 11.9]{ALP23}, there exists a boundary $D$ on $X$ of coregularity $0$ such that $K_X+D\sim 0$, $D=D_1+D_2+D_3+D_4$ where $D_1\sim L$, and $D_2\sim F_1$, $D_3\sim D_4\sim F_2$ are general.

The linear system $|M|$ on $\mathbb{P}$ induces a morphism $\phi\colon \mathbb{P}\to \mathbb{P}^5$ whose image is  a quadric $Q'$ of rank $4$, and the map $\mathbb{P}\to Q'$ is a small resolution. One check that the map $\phi$ induces a contraction $g\colon X\to Y$ such that $D_1\simeq \mathbb{P}^1\times\mathbb{P}^1$ is contracted to $\mathbb{P}^1$ along one of its rulings. Also, the image $Y=\phi(X)$ has degree $5$ in $\mathbb{P}^5$. Note that $Y$ admits the structure of a  quadric bundle $Y\to \mathbb{P}^1$ (cf.  \cite[Proposition 3.9(iii)]{KPr18}). 

We denote by $F'_2$ the strict transform of $F_2$ on $Y$. Observe that $F'_2=\phi(F_2)\simeq F_2$ is a smooth cubic del Pezzo surface, and any divisor equivalent $F'_2$ contains the line $L'$ which is the vertex of the quadric cone $Q'\subset \mathbb{P}^5$. Also, $L'$ is a $(-1)$-curve on any smooth divisor equivalent to $F'_2$.
Let $D'_i$ be the strict transform of $D_i$ on $Y$. Pick a $(-1)$-curve $L_1$ on $D'_3\sim F'_2$ such that $L_1$ is disjoint from~$L'$. We have the following exact sequence 
\[
0 \to N_{L_1/D'_3}\to N_{L_1/Y} \to N_{D'_3/X_1}|_{L_1}\to 0
\]
and $N_{L_1/D'_3}=\oo(-1)$, $N_{D'_3/Y}|_{L_1}=\oo(L')|_{L_1}=\oo$. Thus the above sequence splits: $N_{L_1/Y}=\oo\oplus\oo(-1)$. We will show that a projection from $L_1$ in $\mathbb{P}^5$ induces a birational map $\alpha\colon Y\dashrightarrow \mathbb{P}^3$. Let $\psi\colon {X}_1\to Y$ be the blow-up of $L_1$ on $Y$. 
We obtain the following diagram:
\begin{equation}
\begin{tikzcd}
& & X_1 \ar[rrd, "\pi"] \ar[d, "\psi"] &  & \ \\
X \ar[rr, "\phi"] & & Y \subset \mathbb{P}^5 \ar[rr, dashed, "\alpha"] & & \mathbb{P}^3.
\end{tikzcd}
\end{equation}
We claim that the map $\pi$ is birational. 
First note that $\pi$ is given by the linear system $|\psi^*M_1-E|$ where $M_1$ is the class of hyperplane section in $\mathbb{P}^5$, and $E=\mathrm{Exc}(\psi)$. 
Compute 
\[
(\psi^* M_1-E)^3 = \psi^* M_1^3 - 3(\psi^* M_1)|_E^2 + 3(\psi^* M_1)|_E E|_E - E^3 = 
1. 
\]
where we used $E^3=-c_1(N_{L_1/Y})=1$. Therefore the map $\pi$ is birational. Thus we obtain a pair $(\mathbb{P}^3, D'')$ with reduced boundary which is crepant equivalent to $(X, D)$ and of coregularity $0$. By Theorem \ref{thm-Ducat}, the pair $(\mathbb{P}^3, D'')$ admits a toric model. Hence the pair $(X, D)$ admits a toric model as well.
\end{proof}

\section{Fano threefolds that are blow-ups}
\label{sec-blow-ups}

In this section, we consider Fano threefolds that can be realized as blow-ups of the varieties that we have already considered. First, we consider the case of products:

\begin{proposition}
\label{prop-product-model}
Let $X=\mathbb{P}^1\times S_d$ where $S_d$ is a smooth del Pezzo surface of degree $1\leq d\leq 5$. Then for $d\geq 2$, the variety $X$ admits a toric model. For $d=1$, a general element $X$ in the family admits a toric model.  
\end{proposition}

\begin{proof}
If $S_d$ is either a smooth del Pezzo surface of degree $d\geq 2$ or a general smooth del Pezzo of degree $d=1$, then $S_d$ admits a $1$-complement of coregularity zero~\cite[Theorem 2.1 and Proposition 2.3]{ALP23}.
Hence, by~\cite[Theorem 1.6]{GHK15}, we know that $S_d$ admits a toric model.
Therefore, we can find a $1$-complement $B_d$ of $S_d$ for which $(S_d,B_d)\simeq_{\rm cbir} (\pp^2,H_0+H_1+H_2)$.
In particular, we get 
\[
(\pp^1,\{0\}+\{\infty\})\times (S_d,B_d) 
\simeq_{\rm cbir} 
(\pp^1,\{0\}+\{\infty\}) 
\times 
(\pp^2,H_0+H_1+H_2) 
\simeq_{\rm cbir} 
(\pp^3,H_0+\dots+H_3).
\]
In the last isomorphism, we use the fact that any two toric log Calabi--Yau pairs of the same dimension are crepant birationally equivalent.
Therefore, the variety $X$ admits a toric model.
\end{proof}

We do not consider the case $6\leq d\leq 9$ in the above proposition since $X$ is toric, and this case is treated in Proposition \ref{prop-toric-fanos}.

\begin{proposition}
\label{prop:blow-up}
Let $(Y, D_Y)$ be a lc log CY pair that admits a toric model. Let $\chi\colon X\to Y$ be a birational contraction, and let $(X, D)$ be the log pullback of $(Y, D_Y)$. Assume that $(X, D)$ is a pair. Then $(X, D)$ admits a toric model. 
\end{proposition}
\begin{proof}
This follows from the fact that the pairs $(X, D)$ and $(Y, D_Y)$ are crepant birationally equivalent, and crepant birational equivalence is an equivalence relation.
\end{proof}

We start by analyzing the varieties that can be obtained as a blow-up of Fano threefolds of higher index. 
\begin{lemma}
\label{lem-rho-2-four-cases}
    Let $X:=X_{\rho.m}$ be a Fano threefold which is the blow-up of $Y$, where $X$ and $Y$ are as follows:
    \begin{enumerate}
    \item $\rho.m \in \{2.30, 2.28, 2.27, 2.25, 2.15, 2.12, 2.9, 2.4, 3.18, 3.14, 3.12, 3.6, 4.6\}$, and $Y=\mathbb{P}^3$;
    \item $\rho.m \in \{2.31, 2.29, 2.23, 2.21, 2.17, 2.13, 2.7, 3.20, 3.19, 3.15, 3.10, 4.4\}$, and $Y=Q\subset \mathbb{P}^4$;
    \item $\rho.m \in \{2.26, 2.22, 2.20, 2.14\}$, and $Y=V_5\subset \mathbb{P}^6$; or
    \item $\rho.m \in \{2.19, 2.16, 2.10\}$, and $Y=V_4\subset \mathbb{P}^5$.
\end{enumerate}
Then $X$ admits a toric model.
\end{lemma}

\begin{proof}
In ~\cite[Lemma 12.2, 13.1.1, 13.1.2]{ALP23}
 for each $X$ and $Y$ as in the statement
of the lemma, 
the authors construct a $1$-complement
$(Y,D_Y)$ with ${\rm coreg}(Y,D_Y)=0$
such that the log pull-back of $(Y,D_Y)$ to $X$ is a log pair.
For $Y=\mathbb{P}^3$ and $Y=Q\subset \mathbb{P}^4$, we conclude by applying Proposition \ref{prop:blow-up}, Theorem \ref{thm-Ducat}, and Lemma \ref{lem-quadric-cluster-type}. In the case $Y=V_4$ and $Y=V_5$, we need to check that the boundary divisor $D_Y$ is reducible. This is indeed the case by construction in \cite[Lemma 12.2.3, 12.2.4]{ALP23}.
Then we conclude by applying Proposition \ref{prop:blow-up}, Lemma \ref{V_4-model}, and Lemma \ref{V_5-model}.
\end{proof}

Lemma \ref{lem-rho-2-four-cases} covers all smooth rational Fano threefolds with Picard number $2$ which are non-toric, and not from the families \textnumero\,2.32, 2.24, 2.18 considered in Section \ref{sec-some-fano-threefolds}. 
Now, we treat the remaining families of smooth Fano threefolds with $\rho(X)\geqslant3$, except for the family \textnumero\,3.2 treated in Section \ref{sec-some-fano-threefolds}.

\begin{lemma}\label{lem:Fano-blow-up-8-cases}
    Let $X$ be a rational smooth Fano threefold such that $\rho(X)\geqslant3$.
    Assume that $X$ is the blow-up of some other Fano variety $Y$
    except for $\pp^3,Q,V_4$ or $V_5$.
    Then, the variety $X$ admits a toric model.
\end{lemma}

\begin{proof}
    According to~\cite[Lemma 13.1]{ALP23}, we can construct a $1$-complement $D'$ of $Y$ for which ${\rm coreg}(Y,D')=0$ and its log pull-back to $X$ remains a log pair.
    Thus, we only need to argue that $(Y,D')$
    admits a toric model. We quote the proof of \cite[Lemma 13.1]{ALP23} in case the result is not immediate and then construct a toric model for the corresponding boundary $D'$ on $Y$.
    We will proceed in 8 different cases depending on the isomorphism class of the variety $Y$:\\

\noindent \textit{Case 1:} The variety $Y$ is a smooth divisor of bidegree $(1,1)$ in $\pp^2\times \pp^2$. \\

In this case, $X$ belongs to the families \textnumero\,3.24, 3.13, 3.7, 4.7.
For any boundary $D'$ on $Y$ constructed in 
\cite[Lemma 13.1.3]{ALP23}, we can find a rational curve $C$ of bidegree $(1, 1)$ in $D'$. Using the commutative diagram~\eqref{diagram-2-32}, we obtain a boundary of coregularity zero on a quadric hypersurface in $\mathbb{P}^4$, which gives us a toric model.\\

\noindent \textit{Case 2:} The variety $Y$ is $\pp^1\times \pp^2$.\\

In this case, $X$ belongs to the families \textnumero\,3.22, 3.21, 3.17, 3.8, 3.5, 3.3, 4.5. 
We denote by $\pi_1$ and $\pi_2$ the natural projections from $Y$ to $\mathbb{P}^1$ and $\mathbb{P}^2$, respectively. We consider the following birational map $\phi:Y\dashrightarrow\mathbb{P}^3$, expicitly given by the formula
\[
\phi{([y_0:y_1], [x_0: x_1: x_2])=[x_0y_0 : x_0y_1 : y_1x_1: y_1x_2]}.
\]
Therefore, after blowing up of Y in a curve $C$ of bidegree $(1, 1)$ we obtain the following commutative diagram:
\begin{equation}
\begin{tikzcd}
\text{Bl}_CY \ar[rrd, ] \ar[d, swap, "\pi"] &  & \ \\
Y=\mathbb{P}^1\times \mathbb{P}^2 
 \ar[rr, dashed, "\phi"] & & \mathbb{P}^3
\end{tikzcd}
\end{equation}
where $\pi$ is a blow-up. 
The boundary $D'$ on $Y$ as in 
\cite[Lemma 13.1.4]{ALP23}, can be chosen such that $C\subset D'$. Thus, the log pull-back of $(Y,D')$ to $\pp^3$ gives a log Calabi--Yau pair  $(\pp^3,D'')$ of index one and coregularity zero.
Hence, the pair $(\pp^3,D'')$ and so $(Y,D')$ admit a toric model.\\

\noindent \textit{Case 3:} The variety $Y$ is $\mathbb{P}^1\times \mathbb{P}^1\times\mathbb{P}^1$.\\

In this case, $X$ belongs to the families \textnumero\,4.13, 4.8, 4.3, 4.1. 
Let $L_1$ and $L_2$ be two rational curves of degrees $(0, 0, 1)$ and $(1, 0, 0)$, respectively.
Note that Bl$_{L_i}Y\cong\text{Bl}_{C_1, C_2}\mathbb{P}^1\times\mathbb{P}^2$, where $C_1$ and $C_2$ are disjoint curves and each of them projects isomorphically on $\mathbb{P}^1$. Thus, we have the following commutative diagram:
\begin{equation}
\begin{tikzcd}
\text{Bl}_{L_i}Y \ar[rr, "\alpha_i"] \ar[d, swap, "\pi_i"] &  & \text{Bl}_{C_1, C_2}\mathbb{P}^1\times\mathbb{P}^2 \ar[d, "p"]
\ \\
Y=\mathbb{P}^1\times \mathbb{P}^1\times\mathbb{P}^1
 \ar[rr, dashed, "\phi_i"] & & \mathbb{P}^1\times\mathbb{P}^2.
\end{tikzcd}
\end{equation}
Here $\pi_i$ and $p$ are blow-ups, and $\alpha_i$ is an isomorphism. It is easy to  see that a divisor of degree $(0, 0, 1)$ maps to a divisor of degree $(1, 0)$ under the map $\phi_1$. Analogously, a divisor of degree $(1, 0, 0)$ maps to a divisor of degree $(1, 0)$ under the map $\phi_2$. 

For any boundary $D'$ on $Y$ as in 
\cite[Lemma 13.1.5]{ALP23}, we can choose $L_i$ such that log-pullback of $D'$ under that map $\pi_i$ is a boundary on $\text{Bl}_{L_i}Y$. After pushing down via the morphism $p$, we obtain a log Calabi--Yau pair $(\mathbb{P}^1\times\mathbb{P}^2, B)$ of coregularity zero. Moreover, $B$ contains a divisor of degree $(1, 0)$, hence it contains a rational curve $C$ of degree $(1, 1)$. As in the previous case, we can blow up $\mathbb{P}^1\times\mathbb{P}^2$ in $C$ to obtain a log Calabi--Yau pair on $\mathbb{P}^3$ of coregularity zero. We conclude after applying Theorem $\ref{thm-Ducat}$.\\

\noindent \textit{Case 4:} The variety $Y$ is $V_7$.\\

That is, the variety $Y$ is the blow-up of a point $P_1$ in
$\mathbb{P}^3$. In this case, the variety $X$ belongs to the families \textnumero\,3.23, 3.16, 3.11. 
For any boundary $D'$ on $Y$ as in 
\cite[Lemma 13.1.6]{ALP23}, its crepant transform to $X$ is a boundary $D$ with ${\rm coreg}(X,D)=0$. Furthermore, the push-forward of $D$ to $\mathbb{P}^3$ is a reduced boundary. Thus, we obtain a log Calabi--Yau pair of coregularity zero on $\mathbb{P}^3$, hence $(Y, D')$ admits a toric model.\\

\noindent \textit{Case 5:} The variety $Y$ is $\mathbb{P}_{\mathbb{P}^2}(\oo\oplus\oo(2))$.\\ 

Then $X$ is a Fano threefold in the family \textnumero\,$3.9$. Recall that $X$ is the blow-up of a smooth quartic curve that is contained in a projective plane equivalent to a tautological divisor on $Y$. So we can pick a toric boundary $D'$ on $Y$ containing this curve. Hence $D'$ pulls back to a boundary divisor $D$ on $X$. Since the pair $(Y, D')$ is toric, it has a toric model, and so does $(X, D)$.\\

\noindent\textit{Case 6:} The variety $Y$ is
$\mathbb{P}_{\mathbb{P}^1\times\mathbb{P}^1}(\oo\oplus\oo(1, 1))$.\\

In this case, the variety $X$ is a Fano threefold in the family \textnumero\,$4.2$. Recall that $X$ is the blow-up of an elliptic curve contained in a surface equivalent to a tautological divisor on $Y$. So we can pick a toric boundary $D'$ on $Y$ containing this curve. Hence $D'$ pulls back to a boundary divisor $D$ on $X$. Since the pair $(Y, D')$ is toric, it has a toric model, and so does $(X, D)$. \\

\noindent \textit{Case 7:} The variety $Y$ is a Fano threefold in the family \textnumero\,2.18. \\

In this case, $X$ is a Fano variety \textnumero\,3.4. 
Let $D'$ be a boundary on $Y$ as in 
\cite[Lemma 13.1.7]{ALP23}, see also Lemma \ref{lem-2.18}.
Then its crepant transform to $X$ is a boundary $D$ with ${\rm coreg}(X,D)=0$. By Lemma \ref{lem-2.18}, $(Y, D')$ has a toric model, and so does $(X, D)$.\\

\noindent \textit{Case 8:} The variety $Y$ is a blow-up of a smooth quadric $Q\subset \pp^4$ in a conic $C$.\\

In this case,
$X$ is a Fano threefold in the family \textnumero\,5.1. 
For a boundary $D'$ on $Y$ constructed in 
\cite[Lemma 13.1.10]{ALP23}, its crepant transform to $X$ is a boundary $D$ with ${\rm coreg}(X,D)=0$. Furthermore, the push-forward of $D$ to $Q\subset \mathbb{P}^4$ is a reduced boundary. Thus, we obtain a log Calabi--Yau pair of coregularity zero on $Q$, hence $(Y, D')$ admits a toric model.
\end{proof}

\section{Examples and questions}

In this section, we collect some examples
and questions.
The following three examples show that the three conditions: the generality, the smoothness, and the rationality, are indeed necessary for Theorem~\ref{introthm:3-fold}.

\begin{example}\label{ex:generality}
{\em 
In this example, we exhibit a rational smooth Fano $3$-fold $X$ that does not admit a $1$-complement of birational complexity zero. 
Let $S$ be a smooth del Pezzo surface
of rank one with ${\rm coreg}(S)\geq 1$.
Consider $X:=\pp^1\times S$. 
We argue that $X$ does not have coregularity zero.
If this was the case, then~\cite[Proposition 3.28]{Mor22} would imply that $S$ has coregularity zero. 
Hence, we would obtain a contradiction by~\cite[Proposition 2.3]{ALP23}.
}
\end{example}

\begin{lemma}\label{lem:ex-non-can-sing}
Let $X$ be a projective variety that has a closed point 
$x\in X$ which is an exceptional singularity
that is not canonical.
Then, $X$ does not admit a toric model.
\end{lemma}

\begin{proof}
If $X$ admits a toric model, then there is a $1$-complement $(X,D)$.
Since $x\in X$ is not a canonical singularity
and $K_X+D\sim 0$, we conclude that $D$ passes through $x$.
This contradicts the fact that $x\in X$ is an exceptional singularity (see, e.g.,~\cite[Proposition 2.4]{IP99}).
\end{proof}

\begin{example}\label{ex:smoothness}
{\em 
Let $G\leqslant {\rm GL}(3,\cc)$ be a finite group acting on $\cc^3_{x_0,x_1,x_2}$ such that $G$ has no semi-invariants of degree $\leq 3$ and $\cc^3/G$ has an isolated singularity at the image $x\in X$ of the origin (see, e.g.,~\cite[Theorem 1.2]{MP98}).
The variety $\cc^3/G$ is an orbifold cone over
$\pp^2/H$ where $H$ is the image of $G$ in $\mathbb{P}{\rm GL}(3,\cc)$. 
Note that $\pp^2/H$ is a klt Fano surface 
so it is a rational variety.
This implies that $\cc^3/G$ is a rational variety.

The group $G$ acts on $\pp^3$ by acting on the first three coordinates.
Let $X:=\pp^3/G$ be the quotient which is a klt Fano variety with an isolated exceptional quotient singularity $x\in X$.
Note that the action of $G$ on $\pp^3$ commutes
with the action of $\mathbb{G}_m$ given by
$t\cdot [x_0:x_1:x_2:x_3]=[tx_0:tx_1:tx_2:x_3]$.
Then, the $\mathbb{G}_m$-action descends to $X$
and $\mathbb{G}_m$ acts as the identity on the exceptional divisor of the plt blow-up at $x$.
Let $\mu_d\leqslant \mathbb{G}_m$ be the subgroup of $d$ roots of unity. Let $X_d:=X/\mu_d$.
Let $x_d\in X_d$ be the image of $x\in X$.
Then, for $d$ large enough the minimal log discrepancy of
$X_d$ at $x_d$ is smaller than one.
Further, a finite quotient of an exceptional singularity is again an exceptional singularity.
We conclude that $X_d$ is a rational klt Fano variety with an isolated exceptional quotient singularity that is not canonical.
By Schlessinger's theorem on the rigidity of isolated quotient singularities, we conclude that every deformation $X_\eta$ of $X_d$ admits an isolated exceptional quotient singularity that is not canonical.
Thus, by Lemma~\ref{lem:ex-non-can-sing}, we conclude that no deformation $X_\eta$ of $X_d$ admits a toric model.
}
\end{example}

\begin{example}\label{ex:rationality}
{\em 
A smooth quartic $3$-fold, i.e., a smooth hypersurface $X_4\subset \pp^4$ is a Fano variety that is not rational.
Then, $X_4$ does not admit a toric model.
}
\end{example}

Our work states that a general rational smooth Fano $3$-fold is log rational, i.e., it admits a toric model.
However, as shown by Example~\ref{ex:generality}, the general assumption is indeed needed. 
A natural aim is to complete the picture initiated by this article.

\begin{problem}\label{quest:complete-desc}
Complete the description of smooth Fano $3$-folds with toric models.
\end{problem}

It is worth mentioning that the only examples of rational smooth Fano $3$-folds $X$ for which we know that there is no toric model do not have coregularity zero.
Thus, in order to tackle Problem~\ref{quest:complete-desc}, it is reasonable to consider the following question (which is a special case of~\cite[Conjecture 1.4]{Du24}).

\begin{question}
Let $X$ be a rational smooth Fano $3$-fold.
Let $(X,D)$ be a log Calabi--Yau pair of index one and coregularity zero. Does $(X,D)$ admit a toric model?
\end{question}

We are not aware of general smooth Fano varieties in higher dimensions that do not admit toric models.

\begin{question}\label{quest:toric-model}
Is there a higher-dimensional (of dimension at least $4$) example of a general rational smooth Fano variety
that has no toric model?
\end{question}

In~\cite{EFM24}, the authors define the concept of cluster type Fano varieties.
These are special varieties admitting toric models.
More precisely, a Fano variety $X$ of dimension $n$ is said to be {\em cluster type} if there exists a crepant birational map
$\phi\colon (\pp^n,H_0+\dots+H_n)\dashrightarrow (X,B)$, with $B\geq 0$ for which ${\rm Ex}(\phi)\cap \mathbb{G}_m^n$ has codimension at least two in $\mathbb{G}_m^n$.
Fano varieties of cluster type are tightly related to cluster algebras (see~\cite[Theorem 1.3]{EFM24}).
Every general smooth del Pezzo surface is of cluster type.
Thus, it is natural to classify which general smooth Fano $3$-folds are of cluster type.

\begin{question}\label{quest:cluster-type}
Which smooth rational Fano $3$-folds are of cluster type?
\end{question}

\section{Table of smooth Fano 3-folds with toric models}

In this section, we present a table that collects all families of smooth Fano $3$-folds and describes which of them admit toric models.
In the column ``Brief description", we give a description of the general element of such a family of smooth Fano $3$-folds.
In the column ``Gen. rational", we describe whether a general Fano variety in such a family is rational.
In the column ``Toric model", we mention which statement of the paper is used to deduce that such general rational smooth Fano $3$-fold admits a toric model.

\label{sec-the-table}
\small{
\begin{longtable}{|c|c|p{9cm}|c|c|}
\caption{Toric models of smooth Fano threefolds}\label{table:Fanos}\\
\hline Family & $-K_X^3$ &  Brief description & Gen. rational & Toric model \\
\hline $1.1$ & $2$ & A hypersurface in $\mathbb{P}(1,1,1,1,3)$ of
degree $6$. & No & \\
\hline $1.2$ & $4$ & A hypersurface in $\mathbb{P}^4$ of degree
$4$ or\hfill\break a double cover of a smooth quadric in
$\mathbb{P}^{4}$
branched over a surface of degree $8$. & No &  \\
\hline $1.3$ & $6$ & A complete intersection of a quadric and a
cubic in
$\mathbb{P}^{5}$. & No &  \\
\hline $1.4$ & $8$ & A complete intersection of three quadrics
$\mathbb{P}^{6}$. & No &  \\
\hline $1.5$ & $10$ & A section of
$\mathrm{Gr}(2,5)\subset\mathbb{P}^9$ by
quadric and linear subspace of dimension~$7$ or \hfill\break
a double cover of 1-15 with branch locus an anti-canonical divisor.
& No &  \\
\hline $1.6$ & $12$ & A section of the Hermitian symmetric space
$M=G/P\subset\mathbb{P}^{15}$\hfill\break of type DIII  by a
linear
subspace of dimension~$8$. & Yes & Lemma \ref{lem-x12-toric-model} \\
\hline $1.7$ & $14$ & A section of
$\mathrm{Gr}(2,6)\subset\mathbb{P}^{14}$
by a linear subspace of codimension~$5$. & No & \\
\hline $1.8$ & $16$ & A section of the Hermitian symmetric space
$M=G/P\subset \mathbb{P}^{19}$\hfill\break of type CI  by a linear
subspace
of dimension~$10$ & Yes & Lemma~\ref{lem-x16-toric-model} \\
\hline $1.9$ & $18$ & a section of the $5$-dimensional rational
homogeneous contact\hfill\break manifold
$G_2/P\subset\mathbb{P}^{13}$  by
a linear subspace of dimension~$11$ & Yes & Lemma~\ref{lem-x18-toric-model} \\
\hline $1.10$ & $22$ & a zero locus of three sections of the rank
$3$ vector bundle $\bigwedge^2\mathcal{Q}$,\hfill\break where
$\mathcal{Q}$ is
the universal quotient bundle on $\mathrm{Gr}(3,7)$ & Yes &
Lemma~\ref{lem-x22-toric-model} \\
\hline $1.11$ & $8$ & $V_{1}$ that is a hypersurface in
$\mathbb{P}(1,1,1,2,3)$ of degree $6$ & No &
 \\
\hline $1.12$ & $16$ & $V_{2}$ that is a hypersurface in
$\mathbb{P}(1,1,1,1,2)$ of degree $4$ & No & \\
\hline $1.13$ & $24$ & $V_{3}$ that is a hypersurface in
$\mathbb{P}^{4}$ of degree $3$ & No & \\
\hline $1.14$ & $32$ & $V_{4}$ that is a complete intersection of two
quadrics in $\mathbb{P}^{5}$ & Yes & Lemma~\ref{V_4-model}
\\
\hline $1.15$ & $40$ & $V_{5}$ that is a section of
$\mathrm{Gr}(2,5)\subset\mathbb{P}^9$ by linear subspace of
codimension $3$ & Yes & Lemma~\ref{V_5-model} \\
\hline $1.16$ & $54$ & $Q$ that is a hypersurface in $\mathbb{P}^{4}$
of degree $2$ & Yes & Lemma~\ref{lem-quadric-cluster-type} \\
\hline $1.17$ & $64$ & $\mathbb{P}^{3}$ & Yes & Theorem \ref{thm-Ducat} \\
\hline $2.1$ & $4$ & a blow-up of the Fano threefold $V_1$ along an
elliptic curve\hfill\break that is an intersection of two divisors
from $|-\frac{1}{2}K_{V_1}|$ & No &  \\
\hline $2.2$ & $6$ & a double cover of
$\mathbb{P}^1\times\mathbb{P}^2$
whose branch locus is a divisor of bidegree $(2, 4)$ & No &  \\
\hline $2.3$ & $8$ & the blow-up of the Fano threefold $V_2$ along an
elliptic curve\hfill\break that is an intersection of two divisors
from $|-\frac{1}{2}K_{V_2}|$ & No & \\
\hline $2.4$ & $10$ & the blow-up of $\mathbb{P}^3$ along an
intersection of two cubics & Yes & Lemma \ref{lem-rho-2-four-cases}
\\
\hline $2.5$ & $12$ & the blow-up of the threefold
$V_3\subset\mathbb{P}^4$ along a plane cubic & No & \\
\hline $2.6$ & $12$ & a divisor on
$\mathbb{P}^2\times\mathbb{P}^2$ of bidegree $(2, 2)$
or\hfill\break a double cover of $W$ whose branch locus
is a surface in $|-K_W|$  & No &  \\
\hline $2.7$ & $14$ & the blow-up of $Q$ along the intersection of
two divisors from $|\mathcal{O}_Q (2)|$ & Yes & Lemma
\ref{lem-rho-2-four-cases} \\
\hline $2.8$ & $14$ & a double cover of $V_7$ whose branch locus
is a surface
in $|-K_{V_7}|$ & No &  \\
\hline $2.9$ & $16$ & the blow-up of $\mathbb{P}^3$ along a curve
of degree $7$ and genus~$5$\hfill\break which is an intersection
of cubics & Yes & Lemma \ref{lem-rho-2-four-cases} \\
\hline $2.10$ & $16$ & the blow-up of $V_4\subset\mathbb{P}^5$
along an elliptic curve\hfill\break which is an intersection of
two hyperplane
sections & Yes & Lemma \ref{lem-rho-2-four-cases} \\
\hline $2.11$ & $18$ & the blow-up of $V_3$ along a line & No &
 \\
\hline $2.12$ & $20$ & the blow-up of $\mathbb{P}^3$ along a curve
of degree $6$ and genus~$3$\hfill\break which is an intersection
of cubics & Yes & Lemma \ref{lem-rho-2-four-cases} \\
\hline $2.13$ & $20$ & the blow-up of $Q\subset\mathbb{P}^4$ along
a curve of degree $6$ and genus $2$ & Yes & Lemma \ref{lem-rho-2-four-cases} \\
\hline $2.14$ & $20$ & the blow-up of $V_5\subset\mathbb{P}^6$ along
an elliptic curve\hfill\break which is an intersection of two
hyperplane sections & Yes & Lemma \ref{lem-rho-2-four-cases} \\
\hline $2.15$ & $22$ & the blow-up of $\mathbb{P}^3$ along the
intersection of a quadric and a cubic surface & Yes & Lemma
\ref{lem-rho-2-four-cases} \\
\hline $2.16$ & $22$ & the blow-up of $V_4\subset\mathbb{P}^5$
along a conic & Yes & Lemma \ref{lem-rho-2-four-cases} \\
\hline $2.17$ & $24$ & the blow-up of $Q\subset\mathbb{P}^4$ along
an elliptic curve of degree~$5$ & Yes & Lemma \ref{lem-rho-2-four-cases} \\
\hline $2.18$ & $24$ & a double cover of
$\mathbb{P}^1\times\mathbb{P}^2$ whose branch locus is a divisor of
bidegree $(2, 2)$ & Yes & Lemma \ref{lem-2.18} \\
\hline $2.19$ & $26$ & the blow-up of $V_4\subset\mathbb{P}^5$ along a
line & Yes & Lemma \ref{lem-rho-2-four-cases} \\
\hline $2.20$ & $26$ & the blow-up of $V_5\subset\mathbb{P}^6$
along a twisted cubic & Yes & Lemma \ref{lem-rho-2-four-cases} \\
\hline $2.21$ & $28$ & the blow-up of $Q\subset\mathbb{P}^4$ along
a twisted quartic & Yes & Lemma \ref{lem-rho-2-four-cases} \\
\hline $2.22$ & $30$ & the blow-up of $V_5\subset\mathbb{P}^6$
along a conic & Yes & Lemma \ref{lem-rho-2-four-cases} \\
\hline $2.23$ & $30$ & the blow-up of $Q\subset\mathbb{P}^4$ along a
curve of degree $4$ that is an intersection of a surface in
$|\mathcal{O}_{\mathbb{P}^{4}}(1)\vert_{Q}|$ and a surface in
$|\mathcal{O}_{\mathbb{P}^{4}}(2)\vert_{Q}|$ & Yes & Lemma
\ref{lem-rho-2-four-cases} \\
\hline $2.24$ & $30$ & a divisor on $\mathbb{P}^2\times\mathbb{P}^2$
of bidegree $(1, 2)$ & Yes & Lemma \ref{lem:2.24} \\
\hline $2.25$ & $32$ & the blow-up of $\mathbb{P}^3$ along an elliptic
curve which is an intersection of two quadrics & Yes & Lemma
\ref{lem-rho-2-four-cases} \\
\hline $2.26$ & $34$ & the blow-up of the threefold
$V_5\subset\mathbb{P}^6$ along a line & Yes & Lemma
\ref{lem-rho-2-four-cases} \\
\hline $2.27$ & $38$ & the blow-up of $\mathbb{P}^3$ along a twisted
cubic  & Yes & Lemma \ref{lem-rho-2-four-cases} \\
\hline $2.28$ & $40$ & the blow-up of $\mathbb{P}^3$ along a plane
cubic &Yes & Lemma \ref{lem-rho-2-four-cases} \\
\hline $2.29$ & $40$ & the blow-up of $Q\subset\mathbb{P}^4$ along a
conic & Yes & Lemma \ref{lem-rho-2-four-cases} \\
\hline $2.30$ & $46$ & the blow-up of $\mathbb{P}^3$ along a conic   &
Yes & Lemma \ref{lem-rho-2-four-cases} \\
\hline $2.31$ & $46$ & the blow-up of $Q\subset\mathbb{P}^4$ along a
line & Yes & Lemma \ref{lem-rho-2-four-cases} \\
\hline $2.32$ & $48$ & $W$ that is a divisor on
$\mathbb{P}^2\times\mathbb{P}^2$ of bidegree $(1, 1)$ & Yes &
Lemma~\ref{lem-2.32-Q} \\
\hline $2.33$ & $54$ & the blow-up of $\mathbb{P}^3$ along a line  &
Yes & Proposition~\ref{prop-toric-fanos} \\
\hline $2.34$ & $54$ & $\mathbb{P}^1\times\mathbb{P}^2$  & Yes &
Proposition~\ref{prop-toric-fanos} \\
\hline $2.35$ & $56$ &
$V_7\cong\mathbb{P}(\mathcal{O}_{\mathbb{P}^2}\oplus\mathcal{O}_{\mathbb{P}^2}(1))$
 &Yes & Proposition~\ref{prop-toric-fanos} \\
\hline $2.36$ & $62$ &
$\mathbb{P}(\mathcal{O}_{\mathbb{P}^2}\oplus\mathcal{O}_{\mathbb{P}^2}(2))$
 & Yes & Proposition~\ref{prop-toric-fanos} \\
\hline $3.1$ & $12$ & a double cover of
$\mathbb{P}^1\times\mathbb{P}^1\times\mathbb{P}^1$ branched in a
divisor of type $(2, 2, 2)$ & No & \\
\hline $3.2$ & $14$ & a divisor on a $\mathbb{P}^{2}$-bundle
$\mathbb{P}(\mathcal{O}_{\mathbb{P}^1\times\mathbb{P}^1}\oplus\mathcal{O}_{\mathbb{P}^1\times\mathbb{P}^1}(-1,-1)\oplus\mathcal{O}_{\mathbb{P}^1\times\mathbb{P}^1}(-1,-1))$
such that $X\in|L^{{}\otimes
2}\otimes\mathcal{O}_{\mathbb{P}^{1}\times\mathbb{P}^{1}}(2,3)|$,
where $L$ is the tautological line bundle & Yes & Lemma
\ref{lemma:3-2} \\
\hline $3.3$ & $18$ & a divisor on
$\mathbb{P}^1\times\mathbb{P}^1\times\mathbb{P}^2$ of type $(1,
1, 2)$ & Yes & Lemma~\ref{lem:Fano-blow-up-8-cases}.(2) \\
\hline $3.4$ & $18$ & the blow-up of the Fano threefold $Y$ from the
family \textnumero\,2.18 along a smooth fiber of the
composition $Y\to\mathbb{P}^1\times\mathbb{P}^2\to\mathbb{P}^2$ of the
double cover with the projection & Yes & Lemma~\ref{lem:Fano-blow-up-8-cases}.(7)
\\
\hline $3.5$ & $20$ & the blow-up of $\mathbb{P}^1\times\mathbb{P}^2$
along a curve $C$ of bidegree $(5, 2)$\hfill\break such that the
composition  $C\hookrightarrow\mathbb{P}^1\times\mathbb{P}^2\to\mathbb{P}^2$
is an embedding & Yes & Lemma~\ref{lem:Fano-blow-up-8-cases}.(2) \\
\hline $3.6$ & $22$ & the blow-up of $\mathbb{P}^3$ along a disjoint
union of a line and an elliptic curve of degree~$4$ & Yes & Lemma
\ref{lem-rho-2-four-cases} \\
\hline $3.7$ & $24$ & the blow-up of the threefold $W$ along an
elliptic curve\hfill\break that is an intersection of two  divisors
from $|-\frac{1}{2}K_W|$ & Yes & Lemma \ref{lem:Fano-blow-up-8-cases}.(1)
\\
\hline $3.8$ & $24$ & a divisor in
$|(\alpha\circ\pi_1)^*(\mathcal{O}_{\mathbb{P}^2}(1))\otimes\pi_2^*(\mathcal{O}_{\mathbb{P}^2}(2))|$,
where $\pi_{1}\colon\mathbb{F}_1\times\mathbb{P}^2\to\mathbb{F}_1$
and $\pi_{2}\colon\mathbb{F}_1\times\mathbb{P}^2\to\mathbb{P}^2$ are
projections, and $\alpha\colon\mathbb{F}_1\to\mathbb{P}^2$ is a blow
up of a point & Yes & Lemma
\ref{lem:Fano-blow-up-8-cases}.(2) \\

\hline $3.9$ & $26$ & the blow-up of a cone $W_4\subset\mathbb{P}^6$
over the Veronese surface  $R_4\subset\mathbb{P}^5$\hfill\break with
center in a disjoint union of the vertex and a quartic on
$R_4\cong\mathbb{P}^2$ & Yes & Lemma
\ref{lem:Fano-blow-up-8-cases}.(5) \\

\hline $3.10$ & $26$ & the blow-up of $Q\subset\mathbb{P}^4$ along a
disjoint union of two conics & Yes & Lemma
\ref{lem-rho-2-four-cases} \\

\hline $3.11$ & $28$ & the blow-up of the threefold $V_7$ along an
elliptic curve\hfill\break that is an intersection of  two divisors
from $|-\frac{1}{2}K_{V_7}|$ & Yes & Lemma
\ref{lem:Fano-blow-up-8-cases}.(4) \\

\hline $3.12$ & $28$ & the blow-up of $\mathbb{P}^3$ along a disjoint
union of a line and a twisted cubic & Yes & Lemma
\ref{lem-rho-2-four-cases} \\

\hline $3.13$ & $30$ & the blow-up of
$W\subset\mathbb{P}^2\times\mathbb{P}^2$ along a curve $C$ of
bidegree $(2, 2)$\hfill\break such that
$\pi_{1}(C)\subset\mathbb{P}^2$ and
$\pi_{2}(C)\subset\mathbb{P}^{2}$ are irreducible
conics,\hfill\break where $\pi_{1}\colon W\to\mathbb{P}^2$ and
$\pi_{2}\colon W\to\mathbb{P}^2$ are
natural projections & Yes & Lemma \ref{lem:Fano-blow-up-8-cases}.(1) \\

\hline $3.14$ & $32$ & the blow-up of $\mathbb{P}^3$ along a
disjoint union of a plane cubic curve that is
contained in a plane
$\Pi\subset\mathbb{P}^{3}$ and a point that is not contained in $\Pi$
& Yes & Lemma \ref{lem-rho-2-four-cases} \\

\hline $3.15$ & $32$ & the blow-up of $Q\subset\mathbb{P}^4$ along a
disjoint union of a line and a conic & Yes & Lemma
\ref{lem-rho-2-four-cases} \\

\hline $3.16$ & $34$ & the blow-up of $V_7$ along a proper
transform via the blow-up $\alpha\colon
V_7\to\mathbb{P}^3$ of a twisted cubic
passing through the center of the blow-up $\alpha$ & Yes & Lemma
\ref{lem:Fano-blow-up-8-cases}.(4) \\

\hline $3.17$ & $36$ & a divisor on
$\mathbb{P}^1\times\mathbb{P}^1\times\mathbb{P}^2$ of type $(1,
1, 1)$ & Yes & Lemma \ref{lem:Fano-blow-up-8-cases}.(2) \\

\hline $3.18$ & $36$ & the blow-up of $\mathbb{P}^3$ along a disjoint
union of a line and a conic & Yes & Lemma
\ref{lem-rho-2-four-cases} \\

\hline $3.19$ & $38$ & the blow-up of $Q\subset\mathbb{P}^4$ at two
non-collinear points & Yes & Lemma \ref{lem-rho-2-four-cases} \\

\hline $3.20$ & $38$ & the blow-up of $Q\subset\mathbb{P}^4$ along a
disjoint union of two lines & Yes& Lemma
\ref{lem-rho-2-four-cases} \\

\hline $3.21$ & $38$ & the blow-up of $\mathbb{P}^1\times\mathbb{P}^2$
along a curve of bidegree  $(2, 1)$ & Yes & Lemma
\ref{lem:Fano-blow-up-8-cases}.(2) \\

\hline $3.22$ & $40$ & the blow-up of $\mathbb{P}^1\times\mathbb{P}^2$
along a conic in a fiber of the projection
$\mathbb{P}^{1}\times\mathbb{P}^2\to\mathbb{P}^1$ & Yes & Lemma
\ref{lem:Fano-blow-up-8-cases}.(2) \\

\hline $3.23$ & $42$ & the blow-up of $V_7$ along a proper transform
via the blow-up $\alpha\colon V_7\to\mathbb{P}^3$ of an
irreducible conic passing through the center of the blow-up $\alpha$ &
Yes & Lemma \ref{lem:Fano-blow-up-8-cases}.(4) \\

\hline $3.24$ & $42$ & $W\times_{\mathbb{P}^2}\mathbb{F}_1$, where
$W\to\mathbb{P}^2$ is a $\mathbb{P}^1$-bundle and
$\mathbb{F}_1\to\mathbb{P}^2$ is the
blow-up & Yes & Lemma \ref{lem:Fano-blow-up-8-cases}.(1) \\

\hline $3.25$ & $44$ & the blow-up of $\mathbb{P}^3$ along a disjoint
union of two lines & Yes & Proposition~\ref{prop-toric-fanos} \\

\hline $3.26$ & $46$ & the blow-up of $\mathbb{P}^3$ with center in a
disjoint union of a point and a line  & Yes &
Proposition~\ref{prop-toric-fanos} \\

\hline $3.27$ & $48$ &
$\mathbb{P}^1\times\mathbb{P}^1\times\mathbb{P}^1$  & Yes &
Proposition~\ref{prop-toric-fanos} \\

\hline $3.28$ & $48$ & $\mathbb{P}^1\times\mathbb{F}_1$  & Yes &
Proposition~\ref{prop-toric-fanos} \\

\hline $3.29$ & $50$ & the blow-up of the Fano threefold $V_7$
along a line in $E\cong\mathbb{P}^2$,\hfill\break where $E$ is the
exceptional divisor of the
blow-up $V_7\to\mathbb{P}^3$  & Yes & Proposition~\ref{prop-toric-fanos} \\

\hline $3.30$ & $50$ & the blow-up of $V_7$ along a proper transform
via the blow-up $\alpha\colon V_7\to\mathbb{P}^3$ of a
line that passes through the center of the blow-up $\alpha$  & Yes &
Proposition~\ref{prop-toric-fanos} \\

\hline $3.31$ & $52$ & the blow-up of a cone over a smooth quadric in
$\mathbb{P}^3$ at the vertex & Yes & Proposition~\ref{prop-toric-fanos} \\

\hline $4.1$ & $24$ & divisor on
$\mathbb{P}^1\times\mathbb{P}^1\times\mathbb{P}^1\times\mathbb{P}^1$
of multidegree $(1, 1, 1, 1)$ & Yes & Lemma \ref{lem:Fano-blow-up-8-cases}.(3) \\

\hline $4.2$ & $28$ & the blow-up of the cone over a smooth quadric
$S\subset\nolinebreak\mathbb{P}^3$\hfill\break along a disjoint union
of the vertex and an elliptic curve on $S$ & Yes & Lemma
\ref{lem:Fano-blow-up-8-cases}.(6) \\

\hline $4.3$ & $30$ & the blow-up of
$\mathbb{P}^1\times\mathbb{P}^1\times\mathbb{P}^1$ along a curve of
type $(1, 1, 2)$  & Yes & Lemma \ref{lem:Fano-blow-up-8-cases}.(3) \\

\hline $4.4$ & $32$ & the blow-up of the smooth Fano threefold $Y$
from the family \textnumero\,3.19 along the proper
transform of a conic on the quadric $Q\subset\mathbb{P}^4$\hfill\break
that passes through both centers of the blow-up $Y\to Q$ & Yes &
Lemma \ref{lem-rho-2-four-cases} \\

\hline $4.5$ & $32$ & the blow-up of $\mathbb{P}^1\times\mathbb{P}^2$
along a disjoint union of\hfill\break two irreducible curves of
bidegree $(2, 1)$ and $(1, 0)$ & Yes & Lemma
\ref{lem:Fano-blow-up-8-cases}.(2) \\

\hline $4.6$ & $34$ & the blow-up of $\mathbb{P}^3$ along a disjoint
union of three lines & Yes & Lemma \ref{lem-rho-2-four-cases} \\

\hline $4.7$ & $36$ & the blow-up of
$W\subset\mathbb{P}^2\times\mathbb{P}^2$ along a disjoint union
of\hfill\break two curves of bidegree $(0, 1)$ and $(1, 0)$ & &
Lemma \ref{lem:Fano-blow-up-8-cases}.(1) \\

\hline $4.8$ & $38$ & the blow-up of
$\mathbb{P}^1\times\mathbb{P}^1\times\mathbb{P}^1$ along a curve of
type $(0, 1, 1)$  & Yes & Lemma \ref{lem:Fano-blow-up-8-cases}.(3) \\

\hline $4.9$ & $40$ & the blow-up of the smooth Fano threefold $Y$
from the family
\textnumero\,3.25 along a curve that is contracted
by the blow
up $Y \to\mathbb{P}^3$ & Yes & Proposition~\ref{prop-toric-fanos} \\
\hline $4.10$ & $42$ & $\mathbb{P}^1\times S_7$ & Yes &
Proposition~\ref{prop-toric-fanos} \\
\hline $4.11$ & $44$ & the blow-up of
$\mathbb{P}^1\times\mathbb{F}_1$ along a curve $C\cong\mathbb{P}^1$ such
that $C$ is contained in a fiber $F\cong\mathbb{F}_{1}$ of
the projection
$\mathbb{P}^1\times\mathbb{F}_{1}\to\mathbb{P}^1$ and $C\cdot C=-1$ on
$F$ & Yes & Proposition~\ref{prop-toric-fanos} \\
\hline $4.12$ & $46$ & the blow-up of the smooth Fano threefold
$Y$ from the family \textnumero\,2.33 along two curves that are
contracted by the blow-up
$Y\to\mathbb{P}^3$ & Yes & Proposition~\ref{prop-toric-fanos} \\
\hline $4.13$ & $26$ & the blow-up of
$\mathbb{P}^1\times\mathbb{P}^1\times\mathbb{P}^1$ along a curve of
type $(1, 1, 3)$ & Yes & Lemma \ref{lem:Fano-blow-up-8-cases}.(3) \\
\hline $5.1$ & $28$ & the blow-up of the smooth Fano threefold $Y$
from the family \textnumero\,2.29 along three curves that are
contracted by the
blow-up $Y\to Q$ & Yes & Lemma \ref{lem:Fano-blow-up-8-cases}.(8) \\
\hline $5.2$ & $36$ & the blow-up of the smooth Fano threefold $Y$
in the family $3.25$ along two curves $C_{1}\ne
C_{2}$ that are contracted by the blow-up $\phi\colon
Y\to\nolinebreak\mathbb{P}^3$ and that are contained in the
same exceptional divisor of the blow-up $\phi$  &
Yes & Proposition~\ref{prop-toric-fanos} \\
\hline $5.3$ & $36$ & $\mathbb{P}^1\times S_6$  & Yes &
Proposition~\ref{prop-toric-fanos} \\
\hline $6.1$ & $30$ & $\mathbb{P}^1\times S_5$  & Yes &
Proposition~\ref{prop-product-model} \\
\hline $7.1$ & $24$ & $\mathbb{P}^1\times S_4$  & Yes &
Proposition~\ref{prop-product-model} \\
\hline $8.1$ & $18$ & $\mathbb{P}^1\times S_3$  & Yes &
Proposition~\ref{prop-product-model} \\
\hline $9.1$ & $12$ & $\mathbb{P}^1\times S_2$  & Yes &
Proposition~\ref{prop-product-model} \\
\hline $10.1$ & $6$ & $\mathbb{P}^1\times S_1$  & Yes &
Proposition~\ref{prop-product-model} \\
\hline
\end{longtable}
}

\bibliographystyle{habbrv}
\bibliography{bib}

\vspace{0.5cm}
\end{document}